\documentclass[leqno,11pt]{siamltex}
\usepackage{amsfonts}
\usepackage{amsmath}
\usepackage{amssymb}
\usepackage{subfig}
\usepackage{graphicx,version,color}
\usepackage{bm}
\usepackage{xcolor}
\usepackage{multirow}
\usepackage{booktabs}
\usepackage{algorithm}
\usepackage{algorithmic}

\catcode`\@=11
\@addtoreset{equation}{section} 
\newcommand{\BA}{{\bm A}}
\newcommand{\BI}{{\bm I}}
\newcommand{\BJ}{{\bm J}}
\newcommand{\BR}{{\bm R}}
\newcommand{\BT}{{\bm T}}
\newcommand{\BW}{{\bm W}}
\newcommand{\Ba}{{\bm a}}
\newcommand{\Bb}{{\bm b}}
\newcommand{\Be}{{\bm e}}

\newcommand{\Bu}{{\bm u}}
\newcommand{\Bv}{{\bm v}}
\newcommand{\Bx}{{\bm x}}
\newcommand{\By}{{\bm y}}
\newcommand{\Bphi}{{\bm \phi}}
\newcommand{\Balpha}{{\bm \alpha}}
\newcommand{\td}{{\rm d}}

\newcommand{\tvec}{\text{vec}}

\newtheorem{remark}{Remark}[section]

\title{Deep Neural Networks for Solving Large Linear Systems Arising From High-dimensional Problems}
\author{Yiqi Gu\thanks{Department of Mathematics, The University of Hong Kong, Pokfulam, Hong Kong ({\tt yiqigu@hku.hk, mng@maths.hku.hk}). This work is supported by
Hong Kong Research Grant Council GRF 12300218, 12300519,
17201020, 17300021, C1013-21GF, C7004-21GF and Joint NSFC-RGC N-HKU76921.}
\and Michael K. Ng\footnotemark[1]}

\begin{document}

\maketitle

\begin{abstract}
This paper studies deep neural networks for solving extremely large linear systems arising from high-dimensional problems. Because of the curse of dimensionality, it is expensive to store both the solution and right-hand side vector in such extremely large linear systems. Our idea is to employ a neural network to characterize the solution with much fewer parameters than the size of the solution under a matrix-free setting. We present an error analysis of the proposed method, indicating that the solution error is bounded by the condition number of the matrix and the neural network approximation error. Several numerical examples from partial differential equations, queueing problems, and probabilistic Boolean networks are presented to demonstrate that the solutions of linear systems can be learned quite accurately.
\end{abstract}

\begin{keywords}
very large-scale linear systems; neural networks; partial differential equations; Riesz fractional diffusion; overflow queuing model; probabilistic Boolean networks;
\end{keywords}
\begin{AMS}
65F10; 65F50; 65N22; 68T07; 60K25
\end{AMS}

\pagestyle{myheadings}
\thispagestyle{plain}
\markboth{}{}

\section{Introduction}\label{Sec_introduction}
Linear equations appear widely in applied problems such as partial differential equations (PDEs) and numerical optimization. In physical problems, one usually needs to compute some physical quantity, such as temperature distribution and fluid velocity. Let us suppose the problem is addressed in a $d$-dimensional domain $\Omega\subset\mathbb{R}^d$, in which a grid $\Gamma$ is set up. A grid function $u$ on $\Gamma$ is thereafter introduced to approximate the physical quantity. Then after using a spatial discretization on the problem, $u$ is computed through the following system of linear equations,
\begin{equation}\label{01}
\BA\Bu=\Bb,
\end{equation}
where $\BA\in\mathbb{R}^{\widetilde{M}\times \widetilde{N}}$ ($\widetilde{M},\widetilde{N}\in\mathbb{N}^+$) is a non-singular matrix; $\Bb\in\mathbb{R}^{\widetilde{M}}$ is a given vector; $\Bu=\tvec(u)\in\mathbb{R}^{\widetilde{N}}$ is the vector representation of $u$.

\subsection{Existing methods and difficulties}
Traditional linear solvers, including direct and iterative methods, have been extensively studied for a long time. Let us consider the high-dimensional problem that $\Omega$ is a $d$-dimensional box and \eqref{01} is assembled with tensor product structure on a $N\times N\times\cdots\times N$ ($d$ times) grid, where $N\in\mathbb{N}^+$ is the degree of freedom in every dimension. A series of methods have been developed for the linear system with tensor product structure, such as Krylov subspace method \cite{Kressner2010}, projection method \cite{Ballani2013}, and Cayley transformation \cite{Fan2017}. However, for such linear systems, the numbers of equations and unknowns satisfy $\widetilde{M}=\widetilde{N}=N^d$ that are extremely large even if $d$ is moderately large. In spite of some development concerning large-scale problems \cite{Bai2015,Zhang2019,Niu2020}, the practical performance of many existing methods is still limited by the dimension. For high-dimensional problems with larger $d$, the system size $N^d$ might exceed the machine storage so that even the intermediate solution
cannot be stored entirely in memory.

\subsection{Motivations and contributions}
In recent years, the theory and applications of neural networks (NNs) have been widely studied in a variety of areas, including computer science and applied mathematics. Generally speaking, NNs are a type of function with a nonlinear parametric structure. It has been found in a series of literature that NNs can approximate common functions effectively. In the pioneering work \cite{Cybenko1989,Hornik1989,Barron1993}, the universal approximation theory of two-layer shallow NNs is discussed. In recent research, quantitative information about the NN approximation error was presented for various types of functions, e.g., smooth functions \cite{Lu2017,Liang2017,Yarotsky2017,E2018,Montanelli2019,Suzuki2018,E2019_2,E2022}, piecewise smooth functions \cite{Petersen2018}, band-limited functions \cite{Montanelli2019_2}, continuous functions \cite{Yarotsky2018,Shen2019,Shen2020}.

One remarkable property of NNs is the capability to approximate high-dimensional functions. Many traditional approximation structures, such as finite elements and polynomials, suffer from the curse of dimensionality. Specifically, when approximating a function of $d$-dimensional variable, their error bounds will be $O(J^{-\alpha/d})$, where $J$ is the number of free parameters and $\alpha$ characterizes the regularity of the function. However, NNs can avoid such issues for some special function spaces. A typical example is the Barron space, for which the NN approximation error is either independent of $d$ or increasing with $d$ very slowly \cite{Barron1992,Barron1993,Klusowski2018,E2020,E2020_2,Siegel2020_1,Siegel2020_2,Caragea2020,E2022}. So far, NNs have been applied successfully in solving high-dimensional PDEs and inverse problems \cite{Han2018,Raissi2019,Tipireddy2019,Zang2020,Keller2021}. For high-dimensional PDEs, one successful application is the physics-informed neural networks (PINNs) \cite{Raissi2019}. More precisely, the PDE solution is approximated by a general NN, which is then trained through the minimization of the PDE residual (e.g., the PINN) in the least squares sense.

In this work, we follow the idea of NN approximation for high-dimensional functions adopted in previous applications (e.g., PINNs), and propose a novel NN-based method to solve linear systems \eqref{01}.
Specifically, we realize that in many real-world problems, the physical quantities are continuously distributed. Our method comes from the fact that if the physical quantity $v$ is smooth enough in $\Omega$, NNs with only a small number of parameters can approximate $v$ with the desired accuracy. Meanwhile, if the true solution $\Bu$ of \eqref{01} is a good approximation to $v$, it can also be approximated accurately by such NNs. This allows us to take an NN $\phi$ to characterize $\Bu$ with much fewer parameters than $\widetilde{N}=N^d$. In this method, the unknown numeric vector $\Bu$ in \eqref{01} is replaced with a vector function whose variables are NN parameters. Then the task of solving the large linear system is transformed into solving a new small nonlinear system. In the proposal, the new system is solved by a least squares method under a deep learning framework. This approach is able to solve linear systems of very large sizes that may be difficult for existing methods. Error analysis is also conducted for this method, provided that the true physical quantity is in the Barron space.

Several typical problems are solved using the proposed method in numerical experiments. The first problem is the tensor-structured linear system derived from Poisson's equation using the centered finite difference scheme, in which the system with $(N,d)=(10^4,6)$ is solved effectively. We mention that our method applied to this example is equivalent to PINNs except for a few differences in the setting. Next, we consider linear systems derived from Riesz fractional diffusion \cite{Huang2021}, a fractional differential equation. Beyond PDEs, we apply our method to the discrete problem: overflow queuing problem \cite{Chan1987,Chan1988,Chan1996}, where we succeed in solving such linear systems derived from 10 dimensions (i.e., 10 queues), while earlier work can at most address three dimensions numerically. Finally, as the last example, we solve a $2^d\times2^d$ sparse system from probabilistic Boolean networks. Our method successfully solves a problem with 100 dimensions and $O(10^{30})$ nonzero entries, while previous work \cite{Li2012} merely addresses 30 dimensions and $O(10^4)$ nonzero entries.

\subsection{Organization of paper}
This paper is organized as follows. In Section 2, we review the fully-connected neural networks, the conceptual NN-based method, and the practical algorithms with gradient descent. In Section 3, we estimate the error of the approximate solution under the Barron space hypothesis. Several examples of physical problems are demonstrated in Section 4 to test the performance of the proposed method. Conclusions and discussions about further research work are provided in Section 5.

\section{NN-based Method}
In this section, we introduce the concepts of NNs and explain how to use NNs to approximate the solution of linear systems. In this paper, we use bold fonts to denote matrices and vectors.

\subsection{Fully-connected neural network}\label{Sec_FNN}
Among the many types of NNs, the fully connected neural network (FNN) is the most basic and commonly used in applied mathematics. Mathematically speaking, given $L\in\mathbb{N}^+$ and $M_\ell\in\mathbb{N}^+$ for $\ell=1,\dots,L-1$, where $\mathbb{N}^+$ denotes the set of positive integers, we define the simple nonlinear function $h_{\ell}: \mathbb{R}^{M_{\ell-1}}\rightarrow\mathbb{R}^{M_\ell}$ given by
\begin{equation}
h_{\ell}(\Bx_{\ell}):=\sigma\left(\BW_\ell \Bx_{\ell} + \Bb_\ell \right)
\end{equation}
where $\BW_\ell \in \mathbb{R}^{M_\ell\times M_{\ell-1}}$; $\Bb_\ell \in \mathbb{R}^{M_\ell}$; $\sigma(\By)$ is a given function which is applied entry-wise to a vector $\By$ to obtain another vector of the same size, named activation function. Common activation functions include rectified linear unit (ReLU) $\max\{0,y\}$ and the sigmoidal function $(1+e^{-y})^{-1}$.

Set $M_0=d$, the dimension of the input variable, then an FNN $\phi: \mathbb{R}^{d}\rightarrow\mathbb{R}$ is formulated as the composition of these $L-1$ simple nonlinear functions, namely
\begin{equation}
\phi(\Bx;\theta)=\Ba^\top h_{L-1} \circ h_{L-2} \circ \dots \circ h_{1}(\Bx)\quad \text{for}~\Bx\in\mathbb{R}^d,
\end{equation}
where $\Ba\in \mathbb{R}^{M_{L-1}}$ and $\theta:=\{\Ba,\,\BW_\ell,\,\Bb_\ell:1\leq \ell\leq L-1\}$ denotes the set of all free parameters. Here $M_\ell$ is named as the width of the $\ell$-th layer, and $L$ is named as the depth. The widths and depth characterize the architecture of an FNN. So in fixing $\sigma$, $L$ and $\{M_\ell\}_{\ell=1}^{L-1}$, the FNN architecture is completely determined, but the parameters in $\theta$ are still free. In the following passage, for simplicity, we only consider the architecture with fixed width $M_\ell=M$ for all $\ell=1,\dots,L-1$. We use $\mathcal{F}_{L,M,\sigma}$ to denote the set of all FNNs with depth $L$, width $M$, and activation function $\sigma$.

We can calculate the number of scalar parameters in $\theta$. It is clear that the input layer with $\ell=1$ has $(d+1)M$ scalars, the output layer $\Ba$ has $M$ scalars, and other hidden layers with $\ell=2,\dots,L-1$ totally have $(L-2)M(M+1)$ scalars. So $|\theta|=(d+1)M+M+(L-2)M(M+1)=(L-2)M^2+(d+L)M$.

\subsection{Problem description}
Let us describe the linear system arising from $d$-dimensional problems with a physical or conceptual domain $\Omega$. As a typical example, we assume that $\Omega$ is the $d$-dimensional unit box $[0,1]^d$, and a Cartesian grid is set up on $\Omega$. The following discussion can be easily generalized for other domain geometries and grid settings.
Suppose one aims to determine an unknown real function $v(\Bx)$ for $\Bx\in\Omega$ from a high-dimensional physical problem. One common way is setting a grid on $\Omega$ and determining $v$ on every grid point. Specifically, we let $N\in\mathbb{N}^+$
(the set of positive integers), and let $0\leq x_1<\dots<x_N\leq 1$ be a 1-D grid in $[0,1]$. We use the vector of the form $\Balpha=(\alpha_1,\dots,\alpha_d)$, where each component $\alpha_i$ is an integer in $[1,N]$, to denote a multi-index. Also, we define $\Lambda=\left\{\Balpha:1\leq\alpha_i\leq N ~\text{for}~ i=1,\dots,d\right\}$ as the set of all multi-indices. Then for any $\Balpha$, the column vector $\Bx_\Balpha:=[x_{\alpha_1}~\dots~x_{\alpha_d}]^\top$ represents a Cartesian grid point in $\Omega$, and
\begin{equation}
\Gamma:=\left\{\Bx_\Balpha:\Balpha\in\Lambda\right\}
\end{equation}
is the set of all Cartesian grid points. It is clear that $|\Gamma|=N^d$.

We use a real number $u_\Balpha$ to approximate $v(\Bx_\Balpha)$. By using
computational methods (e.g., finite difference method in solving differential equations), we can derive a linear system concerning $\{u_\Balpha:\Balpha\in\Lambda\}$ from the original physical problem, namely,
\begin{equation}\label{20}
\sum_{\Balpha\in\Lambda} a_{m,\Balpha} u_\Balpha=b_m,\quad\text{for}~m=1,\dots,\widetilde{M},
\end{equation}
where $a_{m,\Balpha}\in\mathbb{R}$ is the coefficient of $u_\Balpha$ in the $m$-th equation (cf. \eqref{01}) and $\widetilde{M}$ is the number of equations. Many linear systems arising from practical problems are matrix-free such that one can directly get the value of $a_{m,\Balpha}$ from $\Balpha$ and $m$ instead of loading it from the storage. In the following discussion, we only consider the linear systems that are matrix-free. Moreover, we assume that the right hand side $\{b_m\}$ can be obtained instantly for specified $m$ when the linear system is being solved, and we do not need to save the entire right hand side in the storage. For example, in solving differential equations, $\{b_m\}$ are the values of a given function at grid points, which can be computed in real-time for $m$ in a small subset of $\{1,\dots,\widetilde{M}\}$. This assumption allows us to implement the memory-saving algorithm proposed in Section \ref{sec:SGD}.

\subsection{A conceptual method}
In most cases, the number of grid points $N$ in every dimension is set large for high resolution. Therefore, one difficulty of solving \eqref{20} is its possibly large size $N^d$ when $d$ is moderately large. At the very worst, $N^d$ exceeds the memory limit, and even a vector in $\mathbb{R}^{N^d}$ cannot be saved entirely in memory. For instance, if one sets $N=10$ grid points in every dimension on a machine with 32G memory, the bytes of a $N^d$ double-precision vector will exceed the memory limit when $d\geq10$. This situation forbids many classical linear solvers, including the matrix-free iterative methods.
We will propose a neural network representation of the unknowns $\{u_\Balpha\}$, which can be viewed as an approximation of the unknown vector with fewer free elements, and hence costs much less storage.

Since the linear system \eqref{20} is derived from a physical problem having a smooth unknown function $v$, it is expected that the unknowns $\{u_\Balpha\}$ are also distributed smoothly on $\Gamma$. Namely, the grid mapping $\chi:\Gamma\rightarrow \mathbb{R}^{N^d}$ defined by $\chi(\Bx_\Balpha)=u_\Balpha$ is spatially smooth (it means the data $\{\Bx_\Balpha,u_\Balpha\}_{\Balpha\in\Gamma}$ can be fit by a function with few high-frequency components). Thanks to the good approximation property for high-dimensional functions (e.g., see \cite{Klusowski2018,Siegel2020_1,Siegel2020_2,Caragea2020,E2022,Shen2020,Lu2021,Shen2021,Shen2021_2}), NNs can be employed to serve as the functioning of $\chi$. Specifically, we introduce an NN $\phi:\mathbb{R}^d\rightarrow\mathbb{R}$, and let $\phi(\Bx_\Balpha;\theta)$ be an approximation of $u_\Balpha$ for all $\Balpha\in\Lambda$. By this setting, the linear system \eqref{20} can be formulated into
\begin{equation}\label{04}
\sum_{\Balpha\in\Lambda} a_{m,\Balpha} \phi(\Bx_\Balpha;\theta)=b_m,\quad\text{for}~m=1,\dots,\widetilde{M},
\end{equation}
where the NN parameter set $\theta$ is the unknown. Note that \eqref{04} is actually a nonlinear system of equations due to the nonlinear structure of NNs. And the number of unknowns $|\theta|=O(M^2L+Md)$ is essentially different from $N^d$, the number of unknowns of the original linear system \eqref{20}. In high-dimensional cases, the former number can be much smaller (see Remark \ref{rmk01}).

For simplicity of notations, we can also formulate \eqref{20} and \eqref{04} as matrix-vector form. Without loss of generality, we assume the unknowns $\{u_\Balpha\}$ in \eqref{20} are ordered lexicographically in a column vector, namely
\begin{equation}\label{24}
\Bu := [u_{(1,1,\dots,1)}~u_{(1,1,\cdots,2)}~\cdots~u_{(N,N,\dots,N)}]^\top\in\mathbb{R}^{N^d},
\end{equation}
and let
\begin{gather}
\BA := \left[\begin{array}{cccc}
         a_{1,(1,1,\dots,1)} & a_{1,(1,1,\dots,2)} & \cdots & a_{1,(N,N,\dots,N)} \\
         \vdots & \vdots & \ddots & \vdots \\
         a_{\widetilde{M},(1,1,\dots,1)} & a_{\widetilde{M},(1,1,\dots,2)} & \cdots & a_{\widetilde{M},(N,N,\dots,N)}
       \end{array}\right]\in\mathbb{R}^{\widetilde{M}\times N^d},\\
\Bb := [b_1~\cdots~b_{\widetilde{M}}]^\top\in\mathbb{R}^{\widetilde{M}},
\end{gather}
then \eqref{20} can be written as
\begin{equation}\label{02}
\BA\Bu=\Bb.
\end{equation}

Similarly, we let
\begin{equation}\label{03}
\Bphi(\theta):=\left[\phi(\Bx_{(1,1,\dots,1)};\theta)~\phi(\Bx_{(1,1,\dots,2)};\theta)~\cdots~\phi(\Bx_{(N,N,\dots,N)};\theta)\right]^\top\in\mathbb{R}^{N^d},
\end{equation}
then \eqref{04} can be written as
\begin{equation}\label{22}
\BA\Bphi(\theta)=\Bb.
\end{equation}

Now we solve the nonlinear system \eqref{22} with $O(M^2L+Md)$ unknowns instead of the original linear system \eqref{02} with $N^d$ unknowns. And the vector $\Bphi(\theta)$ is an approximation of the original solution vector $\Bu$. Usually, the system \eqref{22} does not have an exact solution $\theta$. So we will find the least squares solution of \eqref{22} through the following optimization framework:
\begin{equation}\label{05}
\min_{\theta}~L(\theta)=\frac{1}{\widetilde{M}}\|\BA\Bphi(\theta)-\Bb\|_2^2.
\end{equation}
The loss function $L$ can be decreased via gradient descent methods under the NN learning framework.

To measure the error vector of large linear systems in a fair way, one usually uses the $\ell^2$-norm, which does not increase with the vector size, instead of the Euclidean norm $\|\cdot\|_2$. Specifically, for $\Bu\in\mathbb{R}^{N^d}$, we define
\begin{equation}\label{26}
\|\Bu\|_{\ell^2}:=\sqrt{\frac{\sum_{\Balpha\in\Lambda}|u_\Balpha|^2}{N^d}}=N^{-d/2}\|\Bu\|_2.
\end{equation}

If a feasible solution $\theta_0$ of \eqref{05} is found, it satisfies $L(\theta_0)=\frac{1}{\widetilde{M}}\|\BA\Bphi(\theta_0)-\Bb\|_2^2$, then the error between $\Bphi(\theta_0)$ and the true solution $\Bu$ of \eqref{02} is estimated by
\begin{multline}\label{23}
\|\Bphi(\theta_0)-\Bu\|_{\ell^2}\leq N^{-d/2}\|\Bphi(\theta_0)-\Bu\|_2\leq N^{-d/2}\|\BA^{-1}\|_2\|\BA(\Bphi(\theta_0)-\Bu)\|_2\\
= N^{-d/2}\|\BA^{-1}\|_2\|\BA\Bphi(\theta_0)-\Bb\|_2\leq\|\BA^{-1}\|_2\sqrt{L(\theta_0)},
\end{multline}
provided that $\BA$ is square ($\widetilde{M}=N^d$) and invertible. Therefore the above method finds an approximate solution with the error bounded by the norm of inverse matrix and the resulting minimized loss.

\subsection{Non-unique solutions}
If the linear system has more than one solution, the optimization \eqref{05} may not locate the particular solution we are looking for. One typical problem is the computation of the nontrivial solutions of homogeneous systems. For example, the eigenvectors of $\BA$ can be computed through $(\lambda\BI-\BA)\Bu=0$ given the eigenvalue $\lambda$. Another example is the computation of probability distribution in the overflow queuing problem \cite{Chan1987,Chan1988}.

Now we assume that the linear system
\begin{equation}\label{14}
\BA\Bu=0
\end{equation}
admits nontrivial solutions, and the nullspace is exactly one-dimensional. If we solve \eqref{14} via the unconstrained optimization \eqref{05}, sometimes the trivial solution $\Bu=0$ will be obtained. For instance, due to the implicit regularization of NNs \cite{Cao2021}, gradient descent in deep learning will probably converge to the smoothest solution, i,e. the zero solution. As we do not wish to admit the solution $\Bu=0$, we set constraints on the solution. One simple way is requiring $\|\Bu\|_p=1$ with some $p\in[1,\infty]$, then following \eqref{05} a penalized optimization for \eqref{14} is given by
\begin{equation}\label{15}
\min_{\theta}~L(\theta)=\frac{1}{\widetilde{M}}\|\BA\Bphi(\theta)\|_2^2+\varepsilon^{-1}\left(\|\Bphi(\theta)\|_p-1\right)^2,
\end{equation}
where $\epsilon>0$ is a penalty parameter.

A simpler way is fixing one component of $\Bu$, say $u_{\Balpha_0}=1$ for some multi-index $\Balpha_0$, if $u_{\Balpha_0}$ is known to be nonzero in advance. In this case, the penalized optimization for \eqref{14} is as follows,
\begin{equation}\label{16}
\min_{\theta}~L(\theta)=\frac{1}{\widetilde{M}}\|\BA\Bphi(\theta)\|_2^2+\varepsilon^{-1}\left(\phi(\Bx_{\Balpha_0};\theta)-1\right)^2.
\end{equation}

\subsection{Mini-batch gradient descent algorithm}\label{sec:SGD}
We now describe a class of practical algorithms for the optimization of \eqref{05}. Algorithms for solving \eqref{15} and \eqref{16} can be derived in similar ways.

We rewrite $\BA$ row-wise as follows
\begin{equation}
\BA=[\Ba_1~\cdots~\Ba_{\widetilde{M}}]^\top,
\end{equation}
where $\Ba_k\in\mathbb{R}^{N^d}$ is a row of $\BA$
for $k=1,2,\ldots, \widetilde{M}$.
Then the optimization \eqref{05} can be rewritten as
\begin{equation}\label{06}
\min_{\theta}~L(\theta)=\frac{1}{\widetilde{M}}\sum_{k=1}^{\widetilde{M}}|\Ba_k^\top\Bphi(\theta)-b_k|^2.
\end{equation}
To decrease $L$, a gradient descent method will update $\theta$ in every iteration by
\begin{equation}
\theta\leftarrow\theta-\tau\nabla_\theta L(\theta)
\end{equation}
with the gradient
\begin{equation}\label{07}
\nabla_\theta L(\theta)=\frac{2}{\widetilde{M}}\sum_{k=1}^{\widetilde{M}}\left(\Ba_k^\top\Bphi(\theta)-b_k\right)\cdot \BJ[\Bphi(\theta)]^\top\Ba_k,
\end{equation}
where
\begin{equation}
\BJ[\Bphi(\theta)]=\left[\nabla_\theta \phi(\Bx_{(1,1,\dots,1)};\theta)~\nabla_\theta \phi(\Bx_{(1,1,\dots,2)};\theta)~\dots~\nabla_\theta \phi(\Bx_{(N,N,\dots,N)};\theta)\right]^\top
\end{equation}
is the Jacobian matrix of $\Bphi(\theta)$ and $\tau>0$ is some adaptive learning rate.

However, it is sometimes undesirable to use all $N^d$ terms in \eqref{07} due to computational expense. So in practice, one can use mini-batch gradient descent by selecting a small batch of all terms for training. More precisely, in every iteration, a small subset $\mathcal{S}$ is selected from $\{1,\dots,\widetilde{M}\}$ according to some principles (i.e. random sampling), and $\theta$ is then updated by
\begin{equation}
\theta\leftarrow\theta-\tau\nabla_\theta L_\mathcal{S}(\theta),
\end{equation}
where $L_\mathcal{S}(\theta)=\frac{1}{|\mathcal{S}|}\sum_{k\in\mathcal{S}}|\Ba_k^\top\Bphi(\theta)-b_k|^2$. This algorithm is known as mini-batch gradient descent and is shown in Algorithm \ref{alg01}.

\begin{algorithm}
\caption{NN-based mini-batch gradient descent to solve the linear system $\BA\Bu=\Bb$}
\label{alg01}
\begin{algorithmic}
\REQUIRE hyper-parameters $L$, $M$, $\sigma$, $\{\tau_i\}$; initial guess $\theta_0$.
\ENSURE an approximate solution $\Bphi(\theta)\approx\Bu$.
\STATE initialize $\phi(x;\theta)\in\mathcal{F}_{L,M,\sigma}$ with $\theta\leftarrow\theta_0$
\STATE $i\leftarrow1$
\WHILE{stopping criteria is not satisfied}
\STATE generate $\mathcal{S}\subset\{1,\dots,\widetilde{M}\}$
\STATE evaluate $b_k$ for $k\in\mathcal{S}$
\STATE $\theta\leftarrow\theta-\frac{2\tau_i}{|\mathcal{S}|}
{\displaystyle \sum_{k\in\mathcal{S}} }
\left(\Ba_k^\top\Bphi(\theta)-b_k\right)\cdot\BJ[\Bphi(\theta)]^\top\Ba_k$
\STATE $i\leftarrow i+1$
\ENDWHILE
\STATE return $\Bphi(\theta)$
\end{algorithmic}
\end{algorithm}

We remark that for linear systems with moderately small sizes, choosing $\mathcal{S}=\{1,\dots,\widetilde{M}\}$ is tolerable in the sense of computational cost. In this case, Algorithm 1 computes not only the matrix-vector products, but also the NN-related quantities $\Bphi(\theta)$ and $\BJ[\Bphi(\theta)]$. Therefore, this method might be computationally more expensive than traditional iterative methods, which only conduct matrix-vector multiplications, for small-scale linear systems.

However, Algorithm \ref{alg01} is able to address very large-scale linear systems that traditional methods may not handle. On the one hand, in every iteration, we choose a small subset (indexed by $\mathcal{S}$) of all equations for computation, so only those matrix rows and vector entries that are necessary for the current batch are assessed. We do not need to save the entire matrix or vector in memory (especially in the case that $\BA$ is matrix-free and $b_k$ can be computed in real-time for any $k$). On the other hand, to save the intermediate solution, it suffices to save $\theta$ in memory, whose size can be much smaller than the number of entries of $\Bu$. And we only need to compute a few entries of $\Bphi(\theta)$ at the position where $\Ba_k$ is nonzero, rather than the entire $\Bphi(\theta)$.

The complexity of every iteration in the {\tt while}-loop can be estimated. For a special problem, the dimension $d$ is always fixed, so we do not involve it in the estimation. In usual situations, each $b_k$ can be evaluated with complexity that is independent of $L$, $M$, and $N$ (e.g., evaluating a given function at some grid point), so the evaluation of $b_k$ for all $k\in\mathcal{S}$ costs $O(|\mathcal{S}|)$ FLOPS. Next, we use $N_\text{nz}$ to denote the maximal number of nonzero entries of $\Ba_k$ for all $k\in\{1,\dots,\widetilde{M}\}$. For example, the matrix derived from the centered finite difference scheme on Poisson's equation satisfies $N_\text{nz}=2d+1$ (See Section \ref{sec:Poisson}). By a simple calculation, we can derive that it costs $O(LM^2)$ FLOPS to compute $\phi(\Bx;\theta)$ or $\nabla_\theta\phi(\Bx;\theta)$ for each point $\Bx\in\mathbb{R}^d$, namely, doing forward and back propagations of NNs. Then the complexity of computing the vector $\Bphi(\theta)$ or the matrix $\nabla_\theta\Bphi(\theta)$ is $O(|\theta|LM^2)=O(L^2M^4)$. Hence the complexity of computing $\Ba_k^\top\Bphi(\theta)$ and $\nabla_\theta\Bphi(\theta)^\top\Ba_k$ for all $k\in\mathcal{S}$ is $O(|\mathcal{S}|N_\text{nz}L^2M^4)$. Therefore, every iteration in the {\tt while}-loop costs $O(|\mathcal{S}|N_\text{nz}L^2M^4)$ FLOPS.

It is worth mentioning that the optimization \eqref{05} can also be solved by other mini-batch gradient-based optimizers (e.g., Adam \cite{Kingma2014}). The algorithms using these optimizers can be developed like Algorithm \ref{alg01}. Besides, we remark that the proposed method is not limited to linear systems with tensor product structures. In Algorithm \ref{alg01}, $\BA$ is a general sparse matrix. We do not require $\BA$ to have any more special structures (e.g., banded or Toeplitz matrices). This is essentially different from some existing methods for large-scale linear systems that rely on special properties of the matrix.

\section{Error analysis}
We will give an error analysis of the proposed method in this section. The method is developed based on the approximation property of NNs for smooth functions. Therefore, the analysis should depend on the regularity hypothesis of the target function $v$. A series of recent literatures \cite{Lu2017,Liang2017,Yarotsky2017,E2018,Montanelli2019,Suzuki2018,E2019_2,Petersen2018,Montanelli2019_2,Yarotsky2018,Shen2019,Shen2020} have developed many
results of the NN approximation theory. Here, we consider the NN approximation for Barron-type functions, which is studied extensively in \cite{E2020,E2020_2,Siegel2020_1,Siegel2020_2,Caragea2020,E2022}. The authors show that the approximation error of two-layer FNNs for Barron-type functions is independent of the dimension or increases with it very slowly, hence overcoming the curse of dimensionality. Among various types of Barron spaces, we specifically use the definition described in \cite{E2022}, which corresponds to infinitely wide two-layer ReLU FNNs. The definition and properties of the Barron space/functions will be introduced in this section. Without loss of generality, it is still assumed that $\Omega=[0,1]^d$ in the following discussion.

As discussed in the Section \ref{Sec_FNN}, the class of two-layer (i.e., $L=2$) ReLU FNNs can be reformulated as follows,
\begin{equation}
\mathcal{F}_{2,M,\text{ReLU}}=\left\{\phi:~\phi(\Bx)=\frac{1}{M}\sum_{i=1}^Ma_i\sigma(\Bb_i^\top \Bx+c_i),\quad\forall(a_i,\Bb_i,c_i)\in \mathbb{R}\times\mathbb{R}^d\times\mathbb{R}\right\}.
\end{equation}
Without ambiguity, we specify $\sigma(y)=\max\{0,y\}$ being the ReLU activation function throughout this section.

We consider functions $f_\pi:\mathbb{R}^d\rightarrow\mathbb{R}$ that admit the following representation
\begin{equation}\label{33}
f(\Bx)=\int_{\Omega'}a\sigma(\Bb^\top \Bx+c)\pi(\td a,\td \Bb,\td c)=\mathbb{E}_\pi[a\sigma(\Bb^\top \Bx+c)],\quad\forall \Bx\in\mathbb{R}^d,
\end{equation}
where $\Omega'=\mathbb{R}\times\mathbb{R}^d\times\mathbb{R}$ and $\pi$ is a probability distribution on $(\Omega,\Sigma_{\Omega'})$, with $\Sigma_{\Omega'}$ being a Borel $\sigma$-algebra on $\Omega'$. This representation can be seen as a continuum analog of the two-layer ReLU FNNs in $\mathcal{F}_{2,M,\text{ReLU}}$ as $M\rightarrow\infty$. We remark that in general, there are more than one $\pi$'s such that \eqref{33} is satisfied.

Now let us introduce the Barron space and its norm with respect to $\mathcal{F}_{2,M,\text{ReLU}}$. For functions that admit the representation \eqref{33}, its Barron norm is defined as
\begin{equation}
\|f\|_\mathcal{B}:=\underset{\pi}{\inf}\left(\int_{\Omega'}a^2(\|\Bb\|_1+|c|)^2\pi(\td a,\td \Bb,\td c)\right)^{1/2}=\underset{\pi}{\inf}~\left(\mathbb{E}_\pi[a^2(\|\Bb\|_1+|c|)^2]\right)^{1/2},
\end{equation}
where the infimum is taken over all $\pi$ such that \eqref{33} holds for all $\Bx\in\mathbb{R}^d$. The infimum of the empty set is considered as $+\infty$. The set of all functions with finite Barron norm is denoted by $\mathcal{B}$. Note that $\mathcal{B}$ equipped with the Barron norm is shown to be a Banach space that is named as Barron space \cite{E2022}. Some examples of Barron functions are given in \cite{Barron1993}, including Gaussian density, positive definite functions, smooth functions with high-order derivatives, etc. (see \cite{E2022} for a mathematical connection between the Barron definitions in \cite{Barron1993} and in this paper). The following result characterizes the approximation error of NNs in $\mathcal{F}_{2,M,\text{ReLU}}$ for functions in $\mathcal{B}$.
\begin{lemma}[Theorem 12, \cite{E2020_2}]\label{lem01}
For any $f\in\mathcal{B}$ and any $M\in\mathbb{N}^+$, there exists a two-layer ReLU FNN $\phi$ in $\mathcal{F}_{2,M,\text{ReLU}}$ such that
\begin{equation}\label{10}
\|f-\phi\|_{L^\infty([0,1]^d)}\leq4\|f\|_\mathcal{B}\sqrt{\frac{d+1}{M}}.
\end{equation}
\end{lemma}

However, the solution of the linear system is a grid function defined merely at the set of grid points $\Gamma$ rather than a continuous domain. Therefore, we define the following norm for grid functions based on the above Barron norm; namely, for any $\Bu\in\mathbb{R}^{N^d}$,
\begin{equation}\label{25}
\|\Bu\|_{\mathcal{B},\Gamma}:=\inf\|f\|_\mathcal{B},
\end{equation}
where the infimum is taken over all $f\in\mathcal{B}$ such that $f(\Bx_\Balpha)=u_\Balpha$, $\forall\Balpha\in\Lambda$. Briefly, $\|\Bu\|_{\mathcal{B},\Gamma}$ is the minimal Barron norm among all Barron functions that interpolate $\Bu$ at $\Gamma$. Since $\Gamma$ is finite, there always exists some $C^\infty(\mathbb{R}^{N^d})$ function with a compact support that interpolates $\Bu$ at $\Gamma$, which is a Barron function (see \cite{Barron1993}). So the infimum in \eqref{25} will never be taken on the empty set. And it is trivial to show \eqref{25} is a well-defined norm.

Now let us consider our method, i.e., the NN-based minimization \eqref{05}. The following result shows that the error of our proposed method is bounded by the product of the condition number of the matrix and the NN approximation error. Recall that $\|\cdot\|_{\ell^2}$ is defined by \eqref{26}.
\begin{theorem}\label{thm01}
Suppose $\BA$ in \eqref{05} is square and invertible. Let $\theta^*$ be a minimizer of \eqref{05} with $\phi$ being an FNN of depth $L$ and width $M$. Let $\Bu$ be the solution of the linear system \eqref{02}. Then it satisfies
\begin{equation}\label{21}
\|\Bphi(\theta^*)-\Bu\|_{\ell^2}\leq 4\kappa(\BA)\|\Bu\|_{\mathcal{B},\Gamma}\sqrt{\frac{d+1}{M}},
\end{equation}
where $\kappa(\BA):=\|\BA\|_2\|\BA^{-1}\|_2$ is the condition number of $\BA$.
\end{theorem}
\begin{proof}
Let $\tilde{v}\in\mathcal{B}$ be the function taking the infimum in \eqref{25}. Then by Lemma \ref{lem01}, there exists some $\theta'$ such that $\phi(x;\theta')$ satisfies
\begin{equation}\label{29}
\|\tilde{v}(x)-\phi(x;\theta')\|_{L^\infty(\Omega)}\leq4\|\tilde{v}\|_\mathcal{B}\sqrt{\frac{d+1}{M}}.
\end{equation}
Similar to \eqref{24}, we denote
\begin{equation}\label{11}
\tilde{\Bv}:= [\tilde{v}(\Bx_{(1,1,\dots,1)})~\tilde{v}(\Bx_{(1,1,\dots,2)})~\cdots~\tilde{v}(\Bx_{(N,N,\dots,N)})]^\top\in\mathbb{R}^{N^d}.
\end{equation}
Since $\theta^*$ is the minimizer of \eqref{05}, we have
\begin{multline}\label{27}
\|\Bphi(\theta^*)-\Bu\|_2
\leq\|\BA^{-1}\|_2\|\BA\Bphi(\theta^*)-\Bb\|_2
\leq\|\BA^{-1}\|_2\|\BA\Bphi(\theta')-\Bb\|_2\\
\leq\|\BA^{-1}\|_2\|\BA\|_2\|\Bphi(\theta')-\Bu\|_2
=\kappa(\BA)\|\Bphi(\theta')-\Bu\|_2.
\end{multline}
Note that $\tilde{v}(\Bx_\Balpha)=u_\Balpha$, $\forall\Balpha\in\Lambda$, we have
\begin{equation}\label{28}
\|\Bphi(\theta')-\Bu\|_2
=\|\Bphi(\theta')-\tilde{\Bv}\|_2
\leq N^{d/2}\|\Bphi(\theta')-\tilde{\Bv}\|_\infty
\leq 4N^{d/2}\|\tilde{v}\|_\mathcal{B}\sqrt{\frac{d+1}{M}},
\end{equation}
where the NN approximation \eqref{29} is used. Then \eqref{21} directly follows \eqref{27}, \eqref{28} and the fact that $\|\Bphi(\theta^*)-\Bu\|_{\ell^2}=N^{-d/2}\|\Bphi(\theta^*)-\Bu\|_2$.
\end{proof}

\begin{remark}\label{rmk01}
We can estimate how wide a two-layer NN should be to obtain acceptable accuracy $\epsilon$ under Theorem \eqref{thm01}. As a typical example, we consider the linear system derived from the centered finite difference method in solving a $d$-dimensional Poisson's equation (see Section \ref{sec:Poisson}). Supposing the finite difference scheme has a $p$-th order truncation error, then it can be proved that $\kappa(\BA)\leq CN^p$, where $C$ is independent of $d$ and $N$. From \eqref{21}, it suffices to let $M\geq 16C^2N^{2p}\|u\|_{\mathcal{B},\Gamma}^2(d+1)/\epsilon^2$. In this case, the number of unknowns $|\theta|=(d+2)M \sim O(N^{2p}\|u\|_{\mathcal{B},\Gamma}^2d^2/\epsilon^2)$. This number can be compared with $N^d$, the number of unknowns in the original linear system \eqref{20}. Supposing $\|u\|_{\mathcal{B},\Gamma}$ only increases with $d$ mildly, then $|\theta|$ is less than $N^d$ if $d>2p$, and their difference is much more significant as $d$ increases. This implies our method contains fewer unknowns than traditional linear solvers in high-dimensional problems.
\end{remark}

Note that the error bound \eqref{21} involves the norm $\|\Bu\|_{\mathcal{B},\Gamma}$ and the condition number $\kappa(\BA)$, which both depend on $d$ and $N$. But the relations are usually implicit. For $\|\Bu\|_{\mathcal{B},\Gamma}$, it is small if $\Bu$ has a small interpolant in the sense of Barron norm. And one can infer from \cite{Barron1993} that a function has a small Barron norm if its Fourier transform decays to zero quickly as the frequency increases. Roughly speaking, such functions have ``smooth" images. Consequently, we can simply conclude that $\|\Bu\|_{\mathcal{B},\Gamma}$ is small if the $d$-dimensional mesh of $\Bu$ looks smooth. This is usually true in physical problems because $\Bu$ is an approximation of a smooth physical function. However, we do not have an explicit formula to estimate $\|\Bu\|_{\mathcal{B},\Gamma}$ for specific problems.

In Theorem \ref{thm01}, we specify $\theta^*$ as the global minimizer of \eqref{05} However, In neural network optimization, it is usually difficult to find global minimizers numerically due to nonconvexity of loss  functions. To the best of our knowledge, there is no optimizers that guarantee to identify a global minimizer. Consequently, the overall error is also affected by the optimization error, which is the difference between the theoretical global minimizer $\theta^*$ and the actually found solution $\theta$.

In this paper, we only discuss the approximation property of two-layer shallow networks and take it to figure out the error estimate. We need to mention that the approximation properties of deep networks with $L>2$ have also been studied (\cite{Lu2021,Shen2020,Shen2021,Shen2021_2}). For example, it is proved in \cite{Lu2021} that if $f$ is a $C^s$ smooth function in $[0,1]^d$ with $s\in\mathbb{N}^+$, then there exists a ReLU FNN $\phi$ with width $O(J\log(J))$ and depth $O(K\log(K))$ such that $\|f-\phi\|_{L^\infty([0,1]^d)}\leq O(J^{-2s/d}K^{-2s/d})$ for all $J,K\in\mathbb{N^+}$. By these approximation properties, error estimates for deep networks can also be derived in a similar way as in Theorem \ref{thm01}, and the error bound can be much sharper if $v$ is more special.

\section{Numerical experiments}\label{sec:numerical_experiments}
In this section, several linear systems from physical problems are solved by the proposed NN-based method. Due to the high nonlinearity of NNs concerning their parameters, the least-squares optimization in this method is very nonlinear and, hence, difficult to solve. In practice, we can only find local minimizers rather than global minimizers, which causes a certain amount of optimization error. The optimization error then limits the accuracy of the numerical solution. Consequently, for small linear systems, our NN-based method performs less accurately than well-developed traditional methods (e.g., conjugate gradient method and GMRES), which can achieve errors around machine precision for well-conditioned systems. Nevertheless, the proposed method is capable of extremely large linear systems that traditional methods cannot deal with. In this paper, we choose extremely large systems as numerical examples that are exclusively solved by our method, and we do not have any comparison tests.

Our algorithm is implemented on PyTorch with the CUDA platform. The implementation is not picky for hyper parameters, and the numerical results are usually stable, so we do not spend much time on tuning parameters. Specifically, for neural network initialization, the results do not show any significant differences between the default initialization used by us and other common initialization methods (e.g., Xavier's initialization \cite{Glorot2010} and He's initialization \cite{He2015}). Moreover, we observe in experiments that as long as the batch size is moderately large, enlarging batch sizes will slightly speed up the error decay but can hardly improve the errors of the final numerical solution. In our experiments, suitable batch sizes are chosen according to the time-memory trade-off. For the number of iterations, we empirically set it to make sure the mean of the losses of the last $100$ iterations does not exceed $0.001\%$ of the initial loss. More details about the experiment settings can are listed as follows.

\begin{itemize}
  \item {\em Environment.}
  The method is tested in Python environment. PyTorch library (version 1.10.1) with CUDA toolkit (version 11.3) is utilized for NN implementation and GPU-based parallel computing.
We upload the programs\footnote{The programs can be found in
the website: {\tt https://hkumath.hku.hk/$\sim$mng/mng$
\underline{\mbox{\hspace{2mm}}}$files/dlearn-code.zip}}
for readers and
researchers to generate experimental results.

  \item {\em Optimizer and hyper-parameters.}
  The Algorithm \ref{alg01} is implemented. In each iteration of Algorithm \ref{alg01}, we randomly select a batch of grid points from $\Gamma$ with uniform distribution and take their indices to form $\mathcal{S}$. The learning rates are set to decay from $10^{-3}$ to $10^{-5}$ with linearly decreasing powers; namely, let $\tau_i$ be the learning rate of the $i$-th iteration and $I$ be the maximum number of iterations, then
  \begin{equation}
  \tau_i=10^{-3-\frac{2i}{I}}, \text{for}~i=1,\dots,I.
  \end{equation}
  We remark that we have also tried Adam optimizer, which can obtain smaller errors, but cost slightly more computational time compared with the standard mini-batch gradient descent. Due to the high accuracy of Adam, the results obtained by Adam in comparative tests are not as illustrative as those obtained by mini-batch gradient descent, so we only present the results of the latter.
  \item {\em Stopping criteria.}
  We set sufficiently many iterations for each example, which guarantee that the mean of the losses of the last $100$ iterations is less than $0.001\%$ of the initial loss. This stopping criteria is rigorous enough by our empirical experiences, from which promising conclusions can be derived.
  \item {\em Network setting.}
  FNNs with ReLU activation functions are used in the experiments. We implement the method with various depth $L$ and width $M$ to investigate their effects. The network parameters are initialized using the default setting of PyTorch library; namely,
  \begin{equation}
  a, W_\ell, b_\ell\sim U(-M^{-1/2},M^{-1/2}),\quad \ell=1,\cdots,L-1,
  \end{equation}
  which are generated with uniform distribution.
  \item {\em Testing set and error evaluation.}
  We prescribe a set of $N_\text{test}$ grid points from $\Gamma$ with uniform distribution and take their indices as the testing set, denoted as $\mathcal{T}:=\{(\alpha_1^{n},\cdots,\alpha_d^{n})\}_{n=1}^{N_\text{test}}$. For the examples given true solutions, we define the following $\infty$-error over $\mathcal{T}$ between the numerical solution $\Bphi(\theta)$ and true solution $\Bu$,
  \begin{equation}\label{30}
  \|\Bphi(\theta)-\Bu\|_{\infty,\mathcal{T}}:=\max_{\Balpha\in\mathcal{T}}|u_\Balpha-\phi(\Bx_\Balpha;\theta)|,
  \end{equation}
  and the $\ell^2$-error over $\mathcal{T}$,
  \begin{equation}\label{31}
  \|\Bphi(\theta)-\Bu\|_{\ell^2,\mathcal{T}}:=\left(\frac{1}{N_\text{test}}\sum_{\Balpha\in\mathcal{T}}|u_\Balpha-\phi(x_\Balpha;\theta)|^2\right)^\frac{1}{2}.
  \end{equation}
  And for the examples whose true solutions are unknown, we define the following $\ell^2$-residual over $\mathcal{T}$
  \begin{equation}
  \begin{split}
  \text{Res}_{\ell^2}(\mathcal{T}):=&\left(\frac{1}{N_\text{test}}\sum_{\Balpha\in\mathcal{T}}[\Bb-\BA\Bphi(\theta)]_{\zeta(n)}^2\right)^\frac{1}{2}\\
  =&\left(\frac{1}{N_\text{test}}\sum_{\Balpha\in\mathcal{T}}|b_{\zeta(n)}-\Ba_{\zeta(n)}^\top\Bphi(\theta)|^2\right)^\frac{1}{2},
  \end{split}
  \end{equation}
  where $\zeta(n)=\sum_{k=1}^d (\alpha_k^{n}-1)N^{d-k}+1$ is the position of the multiindex $\Balpha=(\alpha_1^{n},\cdots,\alpha_d^{n})$ in the lexicographical sequence. Note that the solution error is bounded by the product of the matrix inverse norm and the residual, i.e. $\|\Bphi(\theta)-\Bu\|\leq\|\BA^{-1}\|\|\BA\Bphi(\theta)-\Bb\|$, so small residuals imply small errors for well-conditioned problems. For all examples, we set $N_\text{test}=\min\{10^4,N^d\}$, noting that $N^d$ is the number of all grid points so $N_\text{test}$ is at most $N^d$.
  \item {\em Randomness.}
  To check the effect of the uncertainty of results caused by the randomness of the algorithm in the NN initialization and training data, we compute the mean of the errors/residuals of the last $100$ iterations as the ``final" error/residual. Moreover, we repeat each experiment using 10 different random seeds (commands \texttt{numpy.random.seed(n)} for stochastic NumPy subroutines and \texttt{torch.manual\_seed(n)} for PyTorch subroutines) and list the mean and standard deviation of the results (shown as ``mean $\pm$ standard deviation"). It shows in the following resulting tables that the standard deviations are always dominated by the means, so our implementation is numerically stable and convincing.
\end{itemize}

In this section, the matrices arising from numerical examples are quite well-conditioned. The condition numbers are increasing with $N$ and $d$ mildly, and they do not grow with $d$ exponentially (i.e., the curse of dimensionality). Hence we do not encounter an ill-conditioning issue. However, for problems with huge condition numbers, the convergence speed of the gradient descent optimizer will be vastly slowed down. If the condition number is large, one can only reduce the loss function to a small extent in every line search toward the gradient, even with the best step size. Hence it will cost a huge number of iterations to obtain the desired accuracy.

\subsection{Poisson's equation}\label{sec:Poisson}
We consider the Poisson's equation
\begin{equation}\label{09}
\begin{cases}
-\Delta v(\Bx) = f(\Bx),\quad\text{in}~\Omega:=[-1,1]^d,\\
v(\Bx)=0,\quad\text{on}~\partial\Omega,
\end{cases}
\end{equation}
which is an elliptic PDE describing a variety of steady-state physical phenomena. The physical solution of \eqref{09} is set as
\begin{equation}
v(\Bx)=\prod_{i=1}^d\sin(\pi x_i),
\end{equation}
where $x_i$ is the $i$-th component of $\Bx$.

A widely used approach for \eqref{09} is the second-order centered finite difference scheme with uniform grid spacing, which leads to the following matrix with tensor product structure (see \cite{Fan2017} for more details on the structure),
\begin{equation}\label{12}
\BA=\sum_{n=1}^d\underbrace{\BI_N\otimes\cdots\otimes\BI_N}_{n-1~\text{terms}}\otimes\BT\otimes\underbrace{\BI_N\otimes\cdots\otimes\BI_N}_{d-n~\text{terms}},
\end{equation}
where $\otimes$ denotes the Kronecker product; $\BI_N\in\mathbb{R}^{N\times N}$ is the identity matrix; $\BT=\left[T_{i,j}\right]\in\mathbb{R}^{N\times N}$ is given by
\begin{equation}
T_{i,j}=\begin{cases}-2/h^2,\quad j=i,\\1/h^2,\quad j=i\pm1,\\0,\quad\text{else};\end{cases}
\end{equation}
$h=2/(N+1)$ is the grid size.

In the implementation of our method, the batch size $|\mathcal{S}|$ of the mini-batch gradient descent is set to be $\min\{10^4,N^d\}$, and the maximum number of iterations is set to be $5\times10^4$. Note that if the total number of grid points $N^d$ is much larger than the sizes of training and testing sets, then with high probability, the randomly selected training and testing sets will be almost disjoint. But for small systems, the training and testing sets may overlap. And they even coincide if $N^d$ is smaller than $10^4$, in which case the training/testing set consists of all the grid points.

The condition number $\kappa(\BA)$ of this problem is of $O(N^2)$ and is independent of $d$ (\cite{Bottcher1998}). So the theoretical error bound given by \eqref{21} is $O(\|\Bu\|_{\mathcal{B},\Gamma}N^2d^{1/2}M^{-1/2})$, which is increasing with $N$ at least in the second order and decreasing with $M$ in the half order. We will show the error bound is coarse for the actual numerical results in the following tests.

\subsubsection{Test for small sizes}
In the first test, we set the right hand side $\Bb$ as the grid representation of the true physical function $f=-\Delta v=d\pi^2\prod_{i=1}^d\sin(\pi x_i)$. Also, let $\Bv$ be the grid representation of $v$; namely
\begin{equation}\label{32}
\Bv:= [v(\Bx_{(1,1,\dots,1)})~v(\Bx_{(1,1,\dots,2)})~\cdots~v(\Bx_{(N,N,\dots,N)})]^\top\in\mathbb{R}^{N^d}.
\end{equation}
Therefore, the true solution $\Bu$ of the linear system $\BA\Bu=\Bb$ is an approximation of the physical solution $\Bv$, up to a discretization error $O(N^{-2})$ (Theorem 4.2 in \cite{Larsson2003}). On the other hand, our method will numerically solve the linear system, obtaining $\Bphi(\theta)$, which is an approximation of $\Bu$.

The linear system is solved for $d=3$ and various $N$ ($N=6,12,24,48,96$). We implement Algorithm \ref{alg01} with the network size $(L,M)=(3,200)$, in which $\BA$ is applied as a matrix-free operator. Also, we use Matlab backslash to obtain a high-accuracy solution seen as the ``true" solution $\Bu$ for error evaluation. In the test, we note
the size of the linear system is at most $96^3$ by $96^3$, which is still tractable by Matlab sparse solver with high accuracy (the two-norm residual of the Matlab solution is $1.1\times10^{-9}$). The $\infty$-error and $\ell^2$-error between any two
quantities ($\Bphi(\theta)$, $\Bu$ and $\Bv$) over the testing set are presented
in Table \ref{Tab_case1_error_physical_right}. The error curves with respect to $N$ are shown in Figure. \ref{Fig_case1_error_physical_right}. We remark that $\|\cdot\|_{\infty,\mathcal{T}}$ and $\|\cdot\|_{\ell^2,\mathcal{T}}$ follow the definitions in \eqref{30} and \eqref{31}, respectively.

It is observed in Table \ref{Tab_case1_error_physical_right} and in Figure \ref{Fig_case1_error_physical_right} that on the one hand, the solution error $\|\Bphi(\theta)-\Bu\|$ is always below $O(10^{-4})$ for various $N$, though it increases mildly with $N$ since the linear system is larger and hence more difficult to solve as $N$ becomes larger. On the other hand, as expected by the theory, the discretization error $\|\Bu-\Bv\|$ decreases in the rate $O(N^{-2})$ but is still larger than the corresponding solution error $\|\Bphi(\theta)-\Bu\|$.
And hence the error $\|\Bphi(\theta)-\Bv\|$, which characterizes the accuracy of our method in solving the original continuous problem, is dominated by the discretization error. These results imply that for moderately small $N$, the accuracy of solving the linear system is high enough compared with the discretization itself. But if $N$ is further larger, the discretization error is likely to continue decreasing and be dominated by the solution error of the linear system. However, we cannot test larger $N$ due to the memory limitation in the use of Matlab direct solvers.

\begin{table}[h!]
\centering
\subfloat[$\infty$-errors]{
\begin{tabular}{l|ccc}
  \toprule
   & $\|\Bphi(\theta)-\Bu\|_{\infty,\mathcal{T}}$ & $\|\Bu-\Bv\|_{\infty,\mathcal{T}}$ & $\|\Bphi(\theta)-\Bv\|_{\infty,\mathcal{T}}$ \\\hline
$N=6$ & 2.397e-05 $\pm$ 3.299e-05& 6.480e-02 & 6.483e-02 $\pm$ 2.914e-05\\
$N=12$ & 6.114e-05 $\pm$ 2.479e-05& 1.927e-02 & 1.929e-02 $\pm$ 2.327e-05\\
$N=24$ & 1.701e-04 $\pm$ 1.935e-05& 5.249e-03 & 5.346e-03 $\pm$ 1.741e-05\\
$N=48$ & 1.731e-04 $\pm$ 1.745e-05& 1.336e-03 & 1.468e-03 $\pm$ 1.262e-05\\
$N=96$ & 1.774e-04 $\pm$ 1.372e-05& 3.383e-04 & 4.638e-04 $\pm$ 1.281e-05\\\bottomrule
\end{tabular}}\\
\subfloat[$\ell^2$-errors]{
\begin{tabular}{l|ccc}
  \toprule
   & $\|\Bphi(\theta)-\Bu\|_{\ell^2,\mathcal{T}}$ & $\|\Bu-\Bv\|_{\ell^2,\mathcal{T}}$ & $\|\Bphi(\theta)-\Bv\|_{\ell^2,\mathcal{T}}$ \\\hline
$N=6$ & 1.518e-05 $\pm$ 2.129e-05& 3.116e-02 & 3.116e-02 $\pm$ 4.927e-08\\
$N=12$ & 1.704e-05 $\pm$ 1.408e-05& 7.780e-03 & 7.852e-03 $\pm$ 1.724e-07\\
$N=24$ & 5.048e-05 $\pm$ 7.770e-06& 1.971e-03 & 1.985e-03 $\pm$ 2.177e-07\\
$N=48$ & 5.017e-05 $\pm$ 6.368e-06& 4.983e-04 & 5.024e-04 $\pm$ 7.598e-07\\
$N=96$ & 4.978e-05 $\pm$ 4.852e-06& 1.249e-04 & 1.351e-04 $\pm$ 2.032e-06\\\bottomrule
\end{tabular}}\\
\caption{\em Errors for various $N$ in the Poisson's equation with physical right hand sides.}
\label{Tab_case1_error_physical_right}
\end{table}

\begin{figure}
\centering
\subfloat[$\infty$-errors]{
\includegraphics[width=7cm,height=5cm]{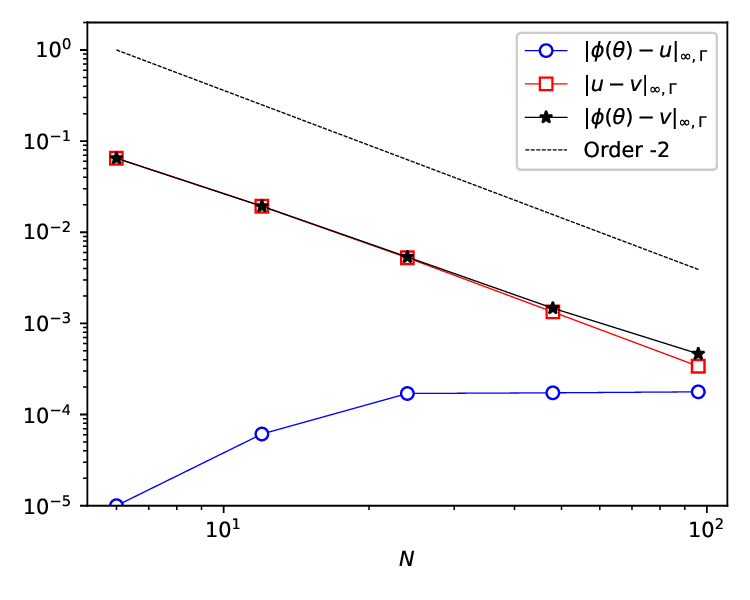}}
\subfloat[$\ell^2$-errors]{
\includegraphics[width=7cm,height=5cm]{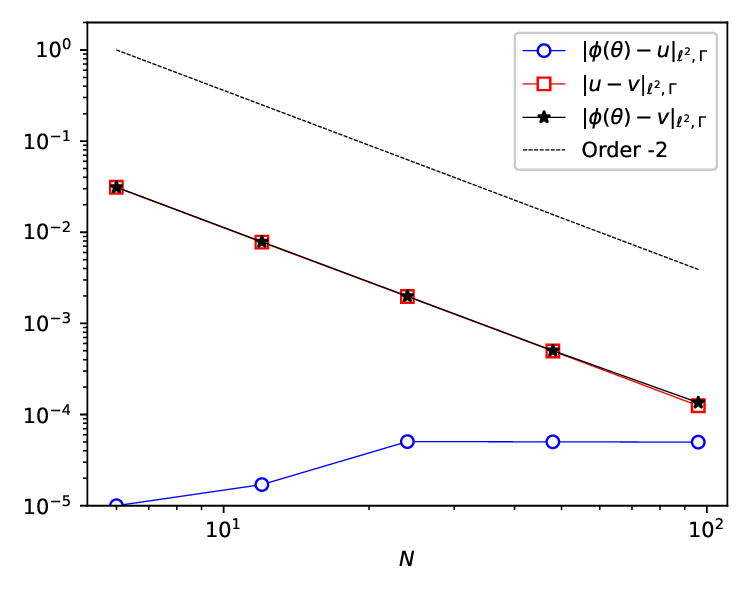}}
\caption{\em Errors versus $N$ in the Poisson's equation with physical right hand sides ($\Bu$ and $\Bphi(\theta)$ are the true and numerical solutions of the linear system, respectively; $\Bv$ is the grid representation of the true solution of the original Poisson's equation).}
\label{Fig_case1_error_physical_right}
\end{figure}

\subsubsection{Test for large sizes}
In this experiment, we turn to solve linear systems of larger sizes in $d$.
Unlike the preceding experiment, we cannot use high-accuracy traditional solvers to find a ``true" solution of the linear system. So we artificially set the true solution as $\Bu=\Bv$ given by \eqref{32}. Namely, we ignore the discretization error and directly take the physical solution as the solution of the linear system. The right hand side is thereafter computed as $\Bb=\BA\Bu$.

We solve $\BA\Bu=\Bb$ for $d=3$ and $6$ using Algorithm \ref{alg01}. The experiment is repeated for various $N$ ($N=10^2,10^3,10^4$) and FNN sizes $(L,M)$ ($L=2,3$ and $M=100,200$). The solution errors over the testing set are shown in Table \ref{Tab_case1_error}. It is observed that the solution error decreases as $L$ or $M$ increases. Also, it is surprising that the solution error does not differ too much for different values of $N$, which implies that the solution error bound given in \eqref{21} is coarse in terms of $N$. Moreover, the running time (i.e., training time for FNN) of the experiment is also reported in Table \ref{Tab_case1_error}. We can find that the running time increases with the network size, but it is almost unchanged for different values of $N$. So the degree of discretization $N$ does not have a strong effect on the efficiency of the method.

It is noted that the number of unknown parameters $|\theta|$ is at most $41200$ when $(L,M)=(3,200)$, which is much less than $N^d$, the size of the linear system. Specifically, in the case that $d=6$ and $N=10^4$, the size $N^d=10^{24}$ is extremely large\footnote{The ratio of $41200$/$10^{24}$ is about 4e-20.} that prevents one using traditional linear solvers, yet the proposed method is still effective, obtaining
$\infty$-errors is at best of $O(10^{-3})$.

We also present the error curve (over the testing set) versus iterations in Figure \ref{Fig_case1_loss} to visualize the dynamics of the optimization. It is observed the error decreases rapidly in the first few iterations and decreases slowly afterward. This means a rough solution can be obtained within much fewer iterations. Also, It can be seen that the error decreases of $L=3$ are more oscillatory than that of $L=2$. The reason is that the loss function \eqref{05} with deeper NNs is highly non-linear in terms of $\theta$, so the optimization is more difficult to be solved using the mini-batch gradient method.

\begin{table}[h!]\small
\centering
\subfloat[$\|\Bphi(\theta)-\Bu\|_{\infty,\mathcal{T}}$ ($d=3$)]{
\begin{tabular}{l|ccc}
  \toprule
  $(L,M)$ & $N=10^2$ & $N=10^3$ & $N=10^4$ \\\hline
$(2,100)$ & 4.578e-03 $\pm$ 1.042e-03& 4.612e-03 $\pm$ 1.094e-03& 4.443e-03 $\pm$ 8.440e-04\\
$(2,200)$ & 1.986e-03 $\pm$ 2.323e-04& 2.034e-03 $\pm$ 1.683e-04& 2.036e-03 $\pm$ 1.606e-04\\
$(3,100)$ & 2.577e-04 $\pm$ 2.449e-05& 2.609e-04 $\pm$ 2.893e-05& 2.646e-04 $\pm$ 3.021e-05\\
$(3,200)$ & 1.818e-04 $\pm$ 1.105e-05& 1.839e-04 $\pm$ 1.194e-05& 1.815e-04 $\pm$ 1.149e-05\\\bottomrule
\end{tabular}}\\
\subfloat[$\|\Bphi(\theta)-\Bu\|_{\infty,\mathcal{T}}$ ($d=6$)]{
\begin{tabular}{l|ccc}
  \toprule
  $(L,M)$ & $N=10^2$ & $N=10^3$ & $N=10^4$ \\\hline
$(2,100)$ & 1.703e-01 $\pm$ 3.143e-02& 1.709e-01 $\pm$ 3.449e-02& 1.718e-01 $\pm$ 2.643e-02\\
$(2,200)$ & 4.471e-02 $\pm$ 5.377e-03& 4.505e-02 $\pm$ 6.190e-03& 4.407e-02 $\pm$ 5.752e-03\\
$(3,100)$ & 1.301e-02 $\pm$ 1.366e-03& 1.393e-02 $\pm$ 1.377e-03& 1.339e-02 $\pm$ 1.074e-03\\
$(3,200)$ & 4.850e-03 $\pm$ 3.314e-04& 4.827e-03 $\pm$ 2.576e-04& 4.960e-03 $\pm$ 4.270e-04\\\bottomrule
\end{tabular}}\\
\subfloat[$\|\Bphi(\theta)-\Bu\|_{\ell^2,\mathcal{T}}$ ($d=3$)]{
\begin{tabular}{l|ccc}
  \toprule
  $(L,M)$ & $N=10^2$ & $N=10^3$ & $N=10^4$ \\\hline
$(2,100)$ & 9.506e-04 $\pm$ 2.932e-04& 9.337e-04 $\pm$ 3.043e-04& 8.818e-04 $\pm$ 2.080e-04\\
$(2,200)$ & 3.400e-04 $\pm$ 3.561e-05& 3.400e-04 $\pm$ 2.936e-05& 3.429e-04 $\pm$ 2.426e-05\\
$(3,100)$ & 6.511e-05 $\pm$ 7.664e-06& 6.434e-05 $\pm$ 6.953e-06& 6.448e-05 $\pm$ 7.084e-06\\
$(3,200)$ & 5.225e-05 $\pm$ 4.369e-06& 5.174e-05 $\pm$ 3.905e-06& 5.092e-05 $\pm$ 3.951e-06\\\bottomrule
\end{tabular}}\\
\subfloat[$\|\Bphi(\theta)-\Bu\|_{\ell^2,\mathcal{T}}$ ($d=6$)]{
\begin{tabular}{l|ccc}
  \toprule
  $(L,M)$ & $N=10^2$ & $N=10^3$ & $N=10^4$ \\\hline
$(2,100)$ & 1.678e-02 $\pm$ 1.690e-03& 1.643e-02 $\pm$ 1.939e-03& 1.638e-02 $\pm$ 1.810e-03\\
$(2,200)$ & 6.739e-03 $\pm$ 3.228e-04& 6.660e-03 $\pm$ 4.591e-04& 6.581e-03 $\pm$ 3.476e-04\\
$(3,100)$ & 2.128e-03 $\pm$ 1.454e-04& 2.084e-03 $\pm$ 1.522e-04& 2.087e-03 $\pm$ 1.504e-04\\
$(3,200)$ & 8.007e-04 $\pm$ 4.432e-05& 7.819e-04 $\pm$ 4.124e-05& 7.807e-04 $\pm$ 4.382e-05\\\bottomrule
\end{tabular}}
\caption{\em Errors for various $d$, $N$, $L$ and $M$ in the Poisson's equation.}
\label{Tab_case1_error}
\end{table}

\begin{table}[h!]\small
\centering
\begin{tabular}{l|ccc|ccc}
  \toprule
  & \multicolumn{3}{c}{$d=3$} & \multicolumn{3}{|c}{$d=6$}\\
  $(L,M)$ & $N=10^2$ & $N=10^3$ & $N=10^4$ & $N=10^2$ & $N=10^3$ & $N=10^4$\\\hline
$(2,100)$ & 5.5e+02 & 5.7e+02 & 5.5e+02 & 1.1e+03 & 1.1e+03 & 1.1e+03 \\
$(2,200)$ & 6.1e+02 & 6.1e+02 & 6.2e+02 & 1.2e+03 & 1.2e+03 & 1.2e+03 \\
$(3,100)$ & 7.1e+02 & 6.9e+02 & 6.9e+02 & 1.3e+03 & 1.4e+03 & 1.3e+03 \\
$(3,200)$ & 1.0e+03 & 1.0e+03 & 1.0e+03 & 2.0e+03 & 2.0e+03 & 2.0e+03 \\\bottomrule
\end{tabular}
\caption{\em Running time (seconds) for various $d$, $N$, $L$ and $M$ in the Poisson's equation.}
\label{Tab_case1_time}
\end{table}

\begin{figure}
\centering
\subfloat[$L=2,M=100$]{
\includegraphics[width=7cm,height=5cm]{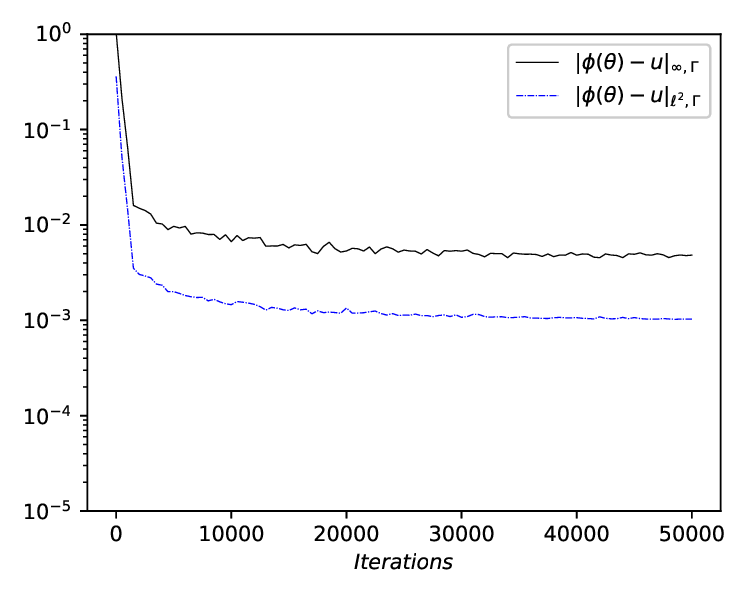}}
\subfloat[$L=2,M=200$]{
\includegraphics[width=7cm,height=5cm]{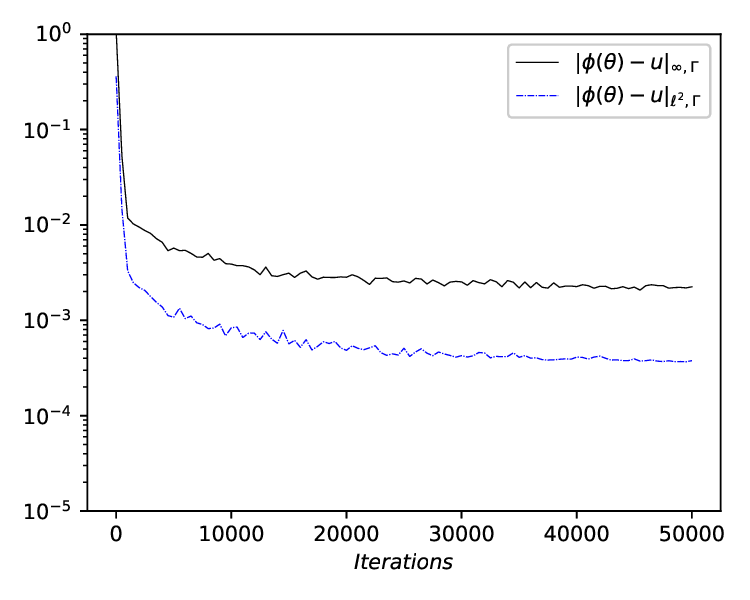}}\\
\subfloat[$L=3,M=100$]{
\includegraphics[width=7cm,height=5cm]{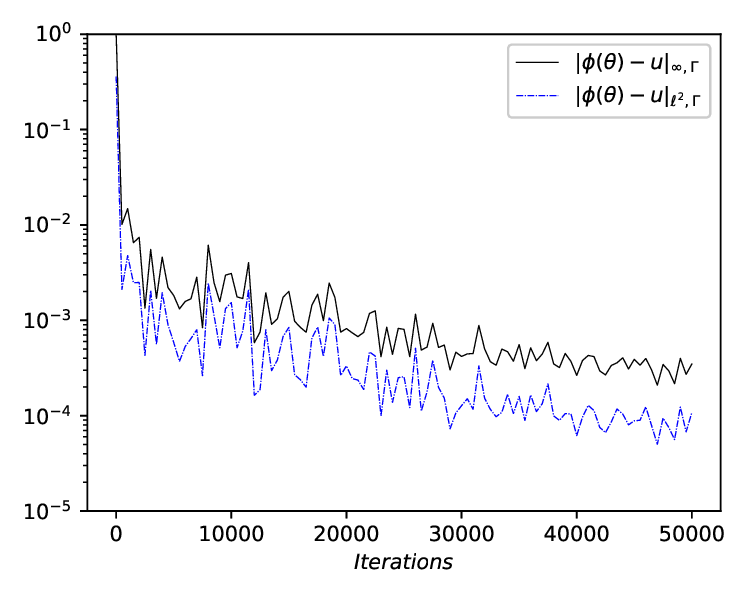}}
\subfloat[$L=3,M=200$]{
\includegraphics[width=7cm,height=5cm]{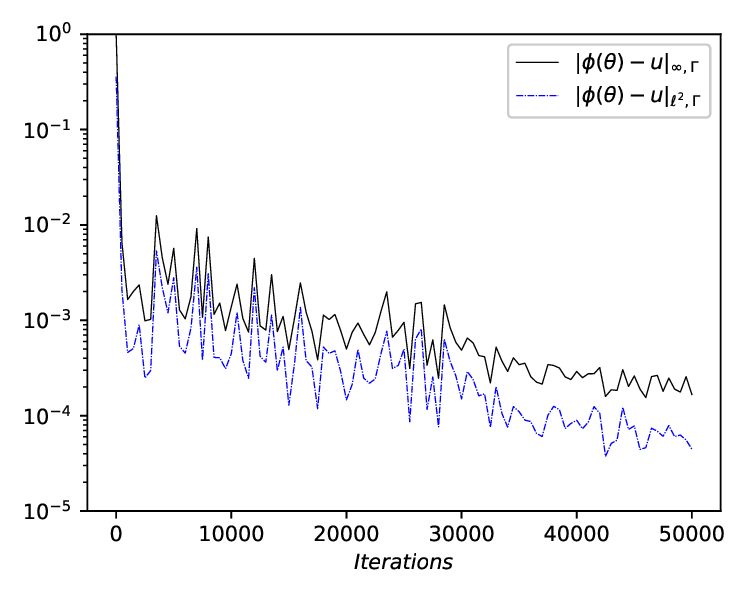}}
\caption{\em Errors versus iterations for various FNN sizes in the Poisson's equation ($d=3$).}
\label{Fig_case1_loss}
\end{figure}

\subsubsection{Error versus $M$}
We conduct a third experiment for Poisson's equation to investigate the relation between the solution error and the network width $M$ and compare the result with the solution error bound given in \eqref{21}. With the same setting as in the second experiment, artificial true solution $\Bu=\Bv$ is still used. We implement our method to solve the problem with $d=3$ and $N=1000$ using $L=2$ and various $M$ ($M=2^3,2^4,\dots,2^{13}$). The $\ell^2$-errors versus $M$ are shown in Figure \ref{Fig_case1_error_vs_M}. It is observed that when $M$ is relatively small ($M\leq2^{11}$ in the figure), the numerical error order is clearly faster than $-\frac{1}{2}$. This implies the error bound in \eqref{21} is slightly coarse in terms of $M$. Indeed, the solution error estimate in \eqref{21} is derived from the NN approximation error $O(M^{-1/2})$ for the class of Barron functions, but some special functions (e.g., the analytic function $v$ in Poisson's equation) might be approximated by NNs more tightly. But the decrease of the error curve becomes flat as $M$ continues increasing and exceeds some threshold ($M\geq2^{11}$ in the figure), in which cases the solution error is dominated by the optimization error of the gradient descent optimizer.
In addition, the running time for various $M$ is shown in Figure \ref{Fig_case1_error_vs_M}. It is clear that when $M$ is moderately large, the running time increases almost linearly with $M$.

\begin{figure}
\centering
\subfloat[$\ell^2$-errors]{
\includegraphics[width=7cm,height=5cm]{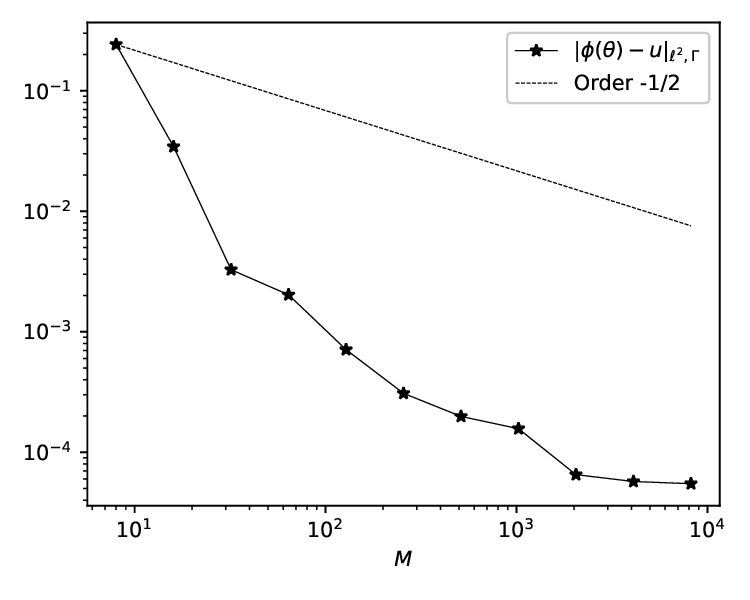}}
\subfloat[Running Time (s)]{
\includegraphics[width=7cm,height=5cm]{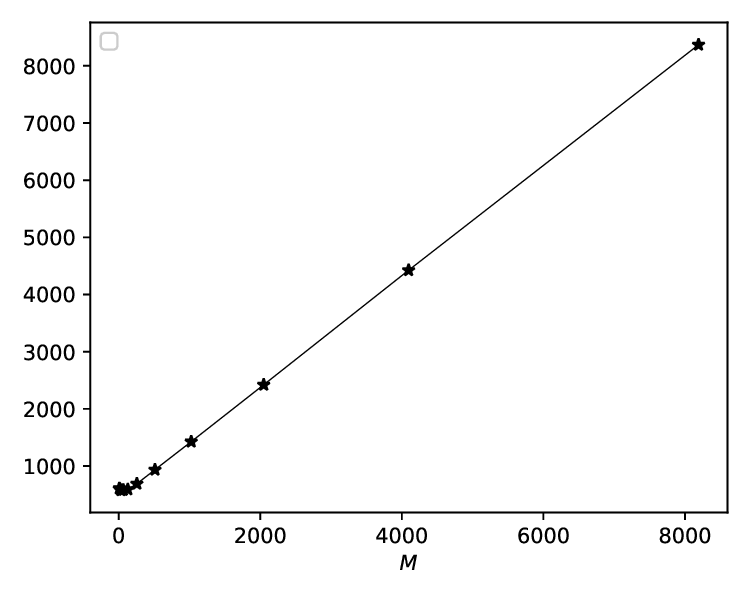}}
\caption{\em $\ell^2$-errors and running time (seconds) versus $M$ in the Poisson's equation.}
\label{Fig_case1_error_vs_M}
\end{figure}

\subsubsection{Error versus $L$}
Similar to the third experiment, we further investigate how the solution error behaves with deeper NNs by a fourth test. We still use the artificial true solution $\Bu=\Bv$. The problem with $d=3$ and $N=1000$ is solved again with various $L$ ($L=2,3,\dots,8$). The network width $M$ is set as $10$ or $20$. The $\ell^2$-errors versus $L$ are shown in Figure \ref{Fig_case1_error_vs_L}. A U-shaped curve is observed that the smallest error is attained when $L=5$. Consequently, for this experiment, our method achieves the best accuracy when $L=5$. It is reasonable to infer that although deeper NNs usually have smaller approximation errors in theory, they may not bring better numerical results because they are more difficult to train in the practical deep learning. In other words, smaller $L$ leads to larger approximation errors, and larger $L$ tends to bring larger optimization errors. We remark that the optimization difficulty is probably a consequence of the vanishing gradient problem, and it could be overcome by residual neural networks \cite{He2016} to some extent. The running time for various $L$ is also reported in Figure \ref{Fig_case1_error_vs_L}. Similar to the relation with $M$, the running time increases almost linearly with the depth $L$.

\begin{figure}
\centering
\subfloat[$\ell^2$-errors]{
\includegraphics[width=7cm,height=5cm]{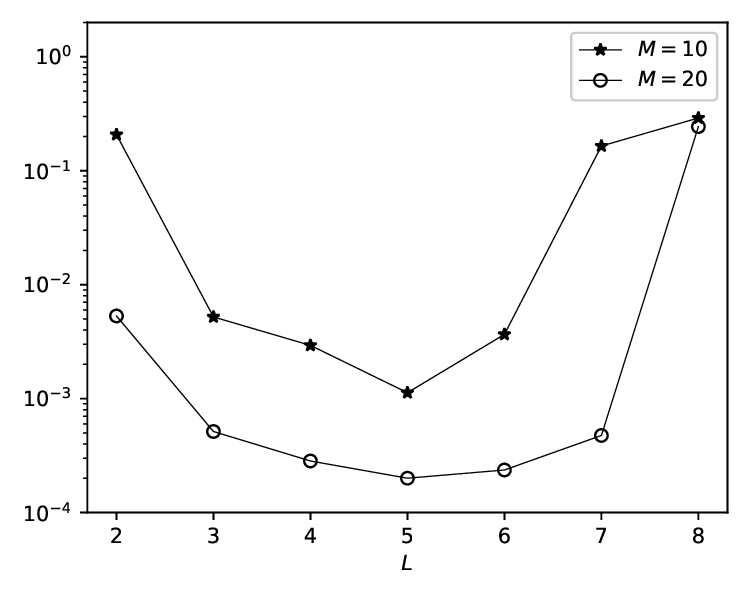}}
\subfloat[Running Time (s)]{
\includegraphics[width=7cm,height=5cm]{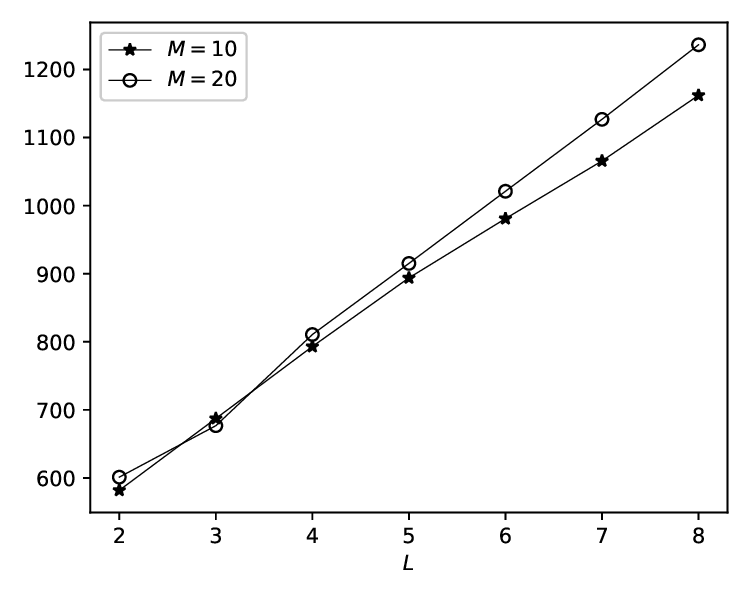}}
\caption{\em $\ell^2$-errors and running time (seconds) versus $L$ in the Poisson's equation.}
\label{Fig_case1_error_vs_L}
\end{figure}

\subsubsection{Some remarks}\label{sec:remark_Poisson}
In this example, we implement our method to solve the linear system derived from Poisson's equation with the finite difference method. We remark that this process is almost equivalent to using PINNs \cite{Raissi2019} to solve the same problem. One slight difference is that we address equidistant grid points, but in PINNs, the distribution of the points can be more general (e.g., uniformly distributed random points). Another is that we consider using finite difference schemes to compute the derivatives numerically, but in PINNs, the differentiation can be either analytically or numerically.

We also note that when the linear system is large, the mini-batch gradient descent cannot take all equations as the training set. For instance, in the test with $N=10^4$ and $d=6$, there are total $10^{24}$ equations, but we only use $5\times10^4$ iterations with batch size $10^4$. So most of the equations are actually not involved in the computation. In spite of a very tiny ratio of the training set to all data points, this method is still effective in finding a solution that is quite accurate globally. This is because the approximate NN can generalize well if the solution is smooth and less oscillatory. The residual of the
non-training equations is also reduced along with the minimization of the training set. It also partially explains why the results do not alter too much even if $N$ increases
quickly as long as the nature of the physical solution is unchanged.

\subsection{Riesz fractional diffusion}
In the second numerical example, we consider the following Riesz fractional diffusion equation:
\begin{equation}\label{13}
-\sum_{n=1}^d c_n\frac{\partial^{\alpha_n} v}{\partial|x_n|^{\alpha_n}} = y(\Bx),\quad\text{in}~\Omega:=[-1,1]^d
\end{equation}
where $c_n>0, 1<\alpha_n<2$ for all $n$ and $\frac{\partial^{\alpha_n} v}{\partial|x_n|^{\alpha_n}}$ is the Riesz fractional derivative (see \cite{Huang2021}). The physical solution of \eqref{13} is set as
\begin{equation}
v(\Bx)=\sin(\sum_{n-1}^dx_n),
\end{equation}
and the corresponding $y$ is therefore given by \eqref{13}.

Employing finite difference method on \eqref{13} leads to a linear system whose matrix $\BA$ is given by
\begin{equation}\label{17}
\BA=\sum_{n=1}^d\underbrace{\BI_N\otimes\cdots\otimes\BI_N}_{n-1~\text{terms}}\otimes\BT^{(n)}\otimes\underbrace{\BI_N\otimes\cdots\otimes\BI_N}_{d-n~\text{terms}},
\end{equation}
with $\BT^{(n)}$ being the Toeplitz matrix
\begin{equation}
\BT^{(n)}:=\left[
\begin{array}{cccccc}
  2t_1^{(n)} & t_0^{(n)}+t_2^{(n)} & t_3^{(n)} & \ddots & t_{N-1}^{(n)} & t_N^{(n)} \\
  t_0^{(n)}+t_2^{(n)} & 2t_1^{(n)} & t_0^{(n)}+t_2^{(n)} & t_3^{(n)} & \ddots & t_{N-1}^{(n)} \\
  \vdots & t_0^{(n)}+t_2^{(n)} & 2t_1^{(n)} & \ddots & \ddots & \vdots \\
  \vdots & \ddots & \ddots & \ddots & \ddots & t_3^{(n)} \\
  t_{N-1}^{(n)} & \ddots & \ddots & \ddots & 2t_1^{(n)} & t_0^{(n)}+t_2^{(n)} \\
  t_N^{(n)} & t_{N-1}^{(n)} & \cdots & \cdots & t_0^{(n)}+t_2^{(n)} & 2t_1^{(n)}
\end{array}
\right]\in\mathbb{R}^{N\times N},
\end{equation}
and $t_0^{(n)}:=\frac{c_n}{2\cos(\frac{\alpha_n\pi}{2})h^{\alpha_n}}$, $t_i^{(n)}=\left(1-\frac{\alpha_n+1}{i}\right)t_{i-1}^{(n)}$ for $1\leq i\leq N$. The condition number of $\BA$ is of $O(N^\alpha)$ with some $1<\alpha<2$, and some preconditioning techniques have been developed (see \cite{Lin2014,Donatelli2016,Moghaderi2017,Donatelli2020,Huang2021}). In prior work, the linear system with at most $d=3$ is solved.

In this experiment, we set $c_n=1$ and $\alpha_n=1.5$ for all $n$. We solve $\BA\Bu=\Bb$ for $d=5,10$ and $N=10$ using Algorithm \ref{alg01}, where the true solution is set as
$\Bu=\Bv$ with $\Bv$ being the grid representation of $v$. We set $|\mathcal{S}|=2\times10^4$ and the number of iterations to be $2\times10^4$. Results are shown in Table \ref{Tab_case2_error}. It is clear that the deeper networks with $L=3$ outperform the shallow ones with $L=2$ in general. It is also noted in the case $d=10$ that the wider networks with $M=200$ performs worse than the narrower one with $M=100$. It implies that larger networks are sometimes more difficult to train than smaller ones, resulting in greater optimization errors. Despite being more accurate in approximation, larger networks may not provide better numerical results in practice.

\begin{table}[h!]\small
\centering
\subfloat[$\|\Bphi(\theta)-\Bu\|_{\infty,\mathcal{T}}$]{
\begin{tabular}{l|cc}
  \toprule
  & $d=5$ & $d=10$ \\\hline
$(2,100)$& 1.059e-01 $\pm$ 2.054e-02& 4.063e-02 $\pm$ 1.206e-02\\
$(2,200)$& 8.205e-02 $\pm$ 3.100e-02& 6.516e-02 $\pm$ 2.599e-02\\
$(3,100)$& 2.947e-03 $\pm$ 4.333e-04& 8.435e-03 $\pm$ 6.248e-03\\
$(3,200)$& 2.349e-03 $\pm$ 2.032e-04& 1.675e-02 $\pm$ 6.920e-03\\\bottomrule
\end{tabular}}\\
\subfloat[$\|\Bphi(\theta)-\Bu\|_{\ell^2,\mathcal{T}}$]{
\begin{tabular}{l|cc}
  \toprule
  & $d=5$ & $d=10$ \\\hline
$(2,100)$& 4.584e-02 $\pm$ 1.382e-02& 9.701e-03 $\pm$ 2.949e-03\\
$(2,200)$& 2.368e-02 $\pm$ 1.393e-02& 1.443e-02 $\pm$ 5.987e-03\\
$(3,100)$& 8.140e-04 $\pm$ 8.405e-05& 1.565e-03 $\pm$ 5.339e-04\\
$(3,200)$& 6.305e-04 $\pm$ 5.490e-05& 2.637e-03 $\pm$ 4.855e-04\\\bottomrule
\end{tabular}}
\caption{\em Errors for various $d$, $L$ and $M$ in the linear systems arising from the Riesz fractional diffusion.}
\label{Tab_case2_error}
\end{table}

\subsection{Overflow queuing model}
The next example is the overflow queuing model proposed in \cite{Chan1987,Chan1988}. Suppose there are $d$ queues with individual queue size $N$, one aims to find nontrivial solutions of the following $N^d\times N^d$ linear system,
\begin{equation}\label{18}
(\BA+\BR)\Bu=0.
\end{equation}
Here the first part $\BA$ has the tensor product structure \eqref{17}, in which $\BT^{(n)}$ is given by
\begin{equation*}
\BT^{(n)}:=\left[\begin{array}{cccccccc}
                    \lambda_n & -\mu_n &  &  &  &  &  &  \\
                    -\lambda_n & \lambda_n+\mu_n & -2\mu_n &  &  &  &  &  \\
                     & -\lambda_n & \lambda_n+2\mu_n & -3\mu_n &  &  &  &  \\
                     &  & \ddots & \ddots & \ddots &  &  &  \\
                     &  &  & -\lambda_n & \lambda_n+s_n\mu_n & -s_n\mu_n &  &  \\
                     &  &  &  & \ddots & \ddots & \ddots &  \\
                     &  &  &  &  & -\lambda_n & \lambda_n+s_n\mu_n & -s_n\mu_n \\
                     &  &  &  &  &  & -\lambda_n & s_n\mu_n
                  \end{array}
\right]\in\mathbb{R}^{N\times N},
\end{equation*}
where $s_n\in\mathbb{N}^+$, $\lambda_n,\alpha\in\mathbb{R}^+$ are physical
parameters, and $\mu_n:=s_n^{-1}(\lambda_n+(N-1)^{-\alpha})$. The second part $\BR=\sum_{m\neq n}\BR_{mn}$ with
\begin{equation}
\BR_{mn}=\bigotimes_{k=1}^{d}(\Be_m\Be_m^\top)^{\delta_{mk}}\BR_m^{\delta_{nk}},
\end{equation}
where $\Be_m$ is the $m$-th unit vector in $\mathbb{R}^N$, $\delta_{mn}$ is the Kronecker delta, and
\begin{equation}
\BR_m:=\lambda_m\left[\begin{array}{ccccc}
                        1 &  &  &  &  \\
                        -1 & 1 &  &  &  \\
                         & \ddots & \ddots &  &  \\
                         &  & -1 & 1 &  \\
                         &  &  & -1 & 0
                      \end{array}
\right].
\end{equation}
The (normalized) nontrivial solution $\Bu$ represents the steady-state probability distribution of the queue system. More precisely, $u_{(i_1,i_2,\cdots,i_d)}$ is the probability that $i_k$ customers are in the $k$-th queue for $k=1,\cdots,d$. It has been shown that $\BA+\BR$ has a one-dimensional nullspace.

This problem has an analogue in the continuous case. It is equivalent to the finite difference approximation to an elliptic PDE with a transport term in a rectangular domain, accompanied with the Neumann boundary condition (except for an oblique derivative condition on one particular side). Despite that the true solution $\Bu$ is not given, we can expect that $\Bu$ is analgue to the PDE solution up to the local truncation error of the finite difference scheme, and hence $\Bu$ can be interpolated by smooth Barron functions so that $\|\Bu\|_{\mathcal{B},\Gamma}$ in the error bound \eqref{21} can be small.

In prior work, one can at most solve the linear system with $d=2$. Due to the one-dimensional nullspace of the matrix, here we use the least-squares model with penalty \eqref{16} to solve \eqref{18} for $d=5,10$, in which the first component of $\Bu$ is fixed as 1 and the penalty parameter $\varepsilon$ is set as 1.0. Since the original problem does not involve any physical domains, a fictitious domain $[0,1]^d$ is introduced for the implementation. We set $N=100$, $\alpha=1$, $\lambda_n=0.01$ and $s_n=8n$ for $n=1,\cdots,d$. The algorithm is implemented with $|\mathcal{S}|=2\times10^4$ and the number of iterations being $2\times10^4$. The residuals of the obtained solutions are listed in Table \ref{Tab_case3_res}.

We are also interested in the actual distribution of the numerical solution. Specifically, in the case $d=10$, $L=2$ and $M=200$, we find $x_\text{max}=x_{(2,6,10,15,18,24,27,32,36,40)}$ is the location where the approximate solution $\Bphi(\theta)$ takes its maximum $2.155$. It is intuitive to find that the $n$-th index of $x_\text{max}$ is equal or close to $s_n/2$. To show a more clear distribution of the solution, we present the 2-D slices passing through $x_\text{max}$ in Figure \ref{Fig_case3_solution}.

Note that the true solutions are unknown. To verify the images in Figure \ref{Fig_case3_solution} are believable, we take a low-dimensional test for comparison. We solve the queuing linear systems with $d=2$ or $3$ by Matlab high-accuracy solvers, and the numerical solutions are accurate since the residual is as small as the machine precision. In the low-dimensional cases, we have observed the same property as in the previous high-dimensional cases; namely, the numerical solution attains its maximum at the position $s_n/2$ for the $n$-th index. So we believe that the probability distribution shown in Figure \ref{Fig_case3_solution} is a good simulation, at least in the sense of locating the maxima.

\begin{table}[h!]\small
\centering
\begin{tabular}{l|cc}
  \toprule
  $(L,M)$ & $d=5$ & $d=10$ \\\hline
$(2,100)$& 7.625e-04 $\pm$ 3.078e-05& 1.268e-03 $\pm$ 2.582e-04\\
$(2,200)$& 7.135e-04 $\pm$ 1.729e-05& 1.279e-03 $\pm$ 1.555e-04\\
$(3,100)$& 6.519e-04 $\pm$ 8.867e-06& 1.073e-03 $\pm$ 3.346e-05\\
$(3,200)$& 6.492e-04 $\pm$ 1.422e-05& 1.100e-03 $\pm$ 2.585e-05\\\bottomrule
\end{tabular}
\caption{\em Residuals for various $d$, $L$ and $M$ in the queuing problem.}
\label{Tab_case3_res}
\end{table}

\begin{figure}
\centering
\subfloat[$(x_1,x_2)$-slice]{
\includegraphics[scale=0.5]{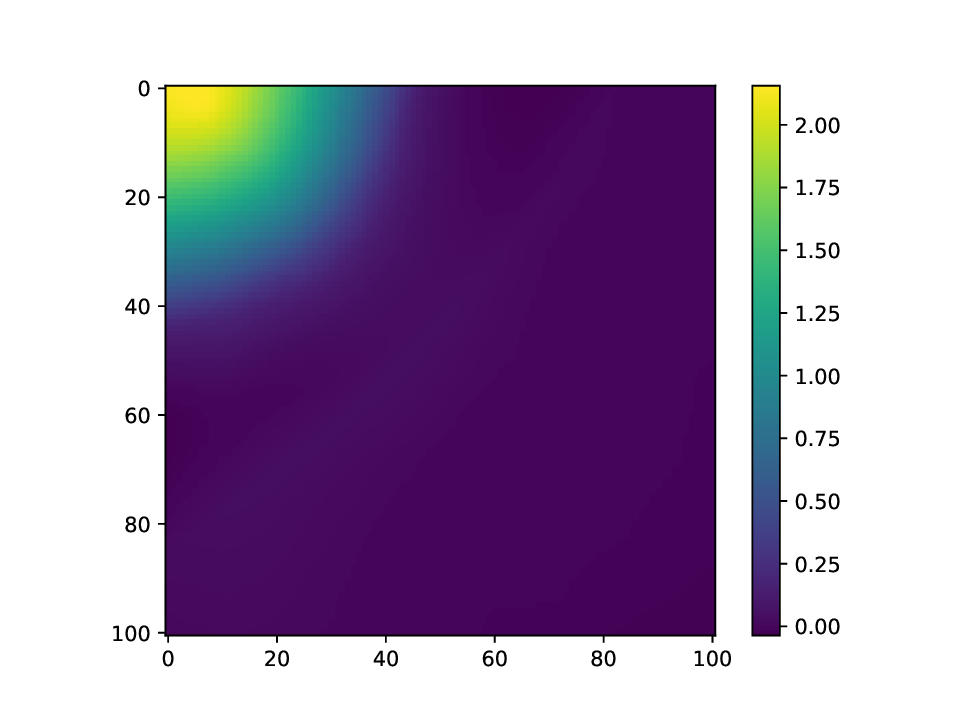}}
\subfloat[$(x_3,x_4)$-slice]{
\includegraphics[scale=0.5]{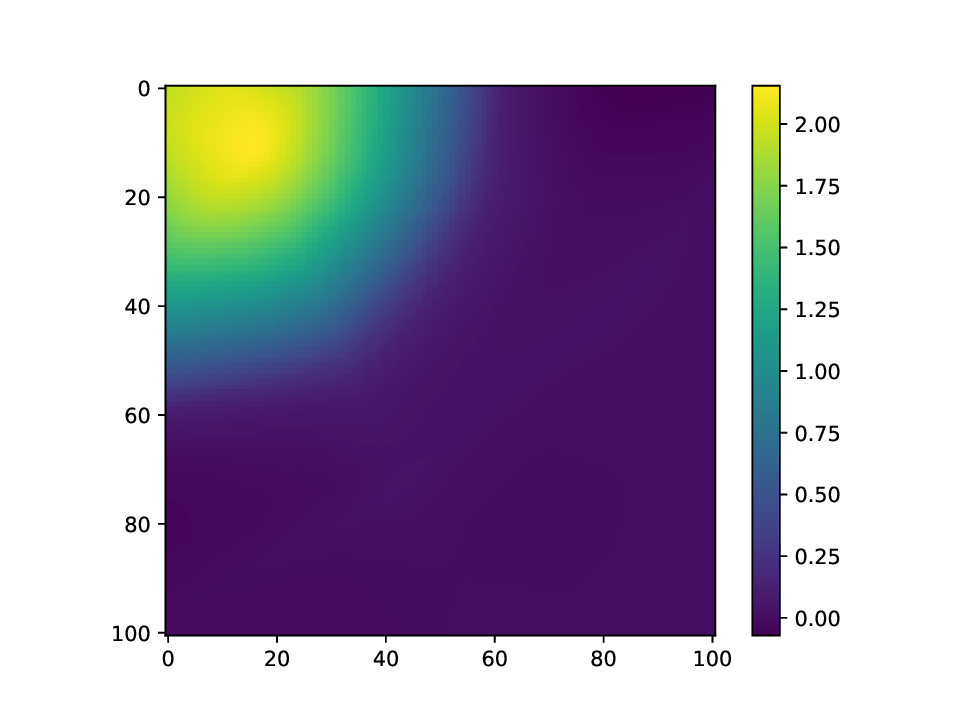}}\\
\subfloat[$(x_5,x_6)$-slice]{
\includegraphics[scale=0.5]{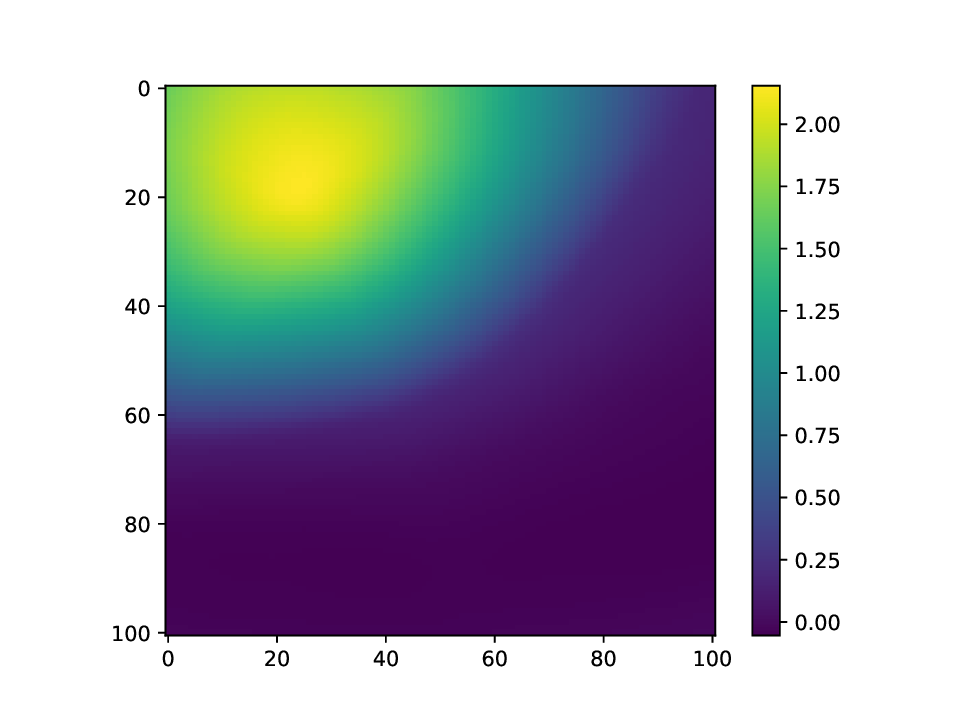}}
\subfloat[$(x_7,x_8)$-slice]{
\includegraphics[scale=0.5]{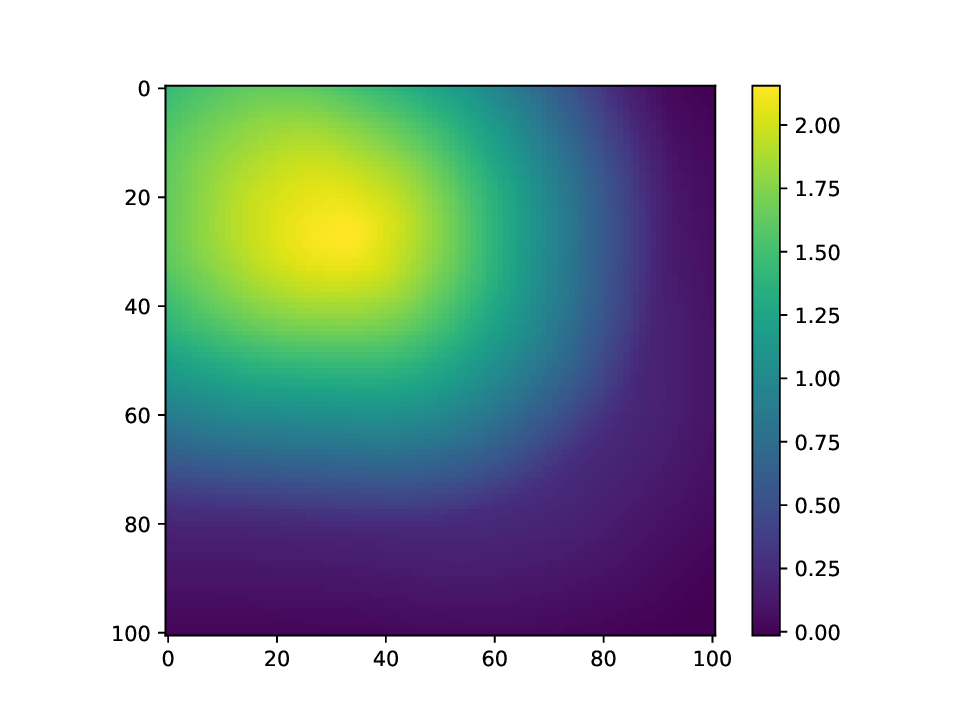}}
\caption{\em 2-D slices of the numerical solution $\phi(x,\theta)$ passing through the maximal point $x_\text{max}$ in the queuing problem ($d=10, L=3, M=200$).}
\label{Fig_case3_solution}
\end{figure}

\subsection{Probabilistic Boolean networks}
Let us consider the steady-state probability distribution of probabilistic Boolean networks, which are widely applied in real-world problems such as genomic signal processing \cite{Li2012}. In this problem, one aims to find the eigenvector associated with the principle eigenvalue 1 of the transition probability matrix, and the normalized eigenvector exactly represents the steady-state probability distribution. The transition probability matrix is of size $2^d$ by $2^d$, where $d$ is the number of genes.

In our experiment, we generate a sparse Toeplitz matrix $\widetilde{\BT}=\left[t_{ij}\right]\in\mathbb{R}^{2^d\times 2^d}$ by
\begin{equation}
t_{ij}=\begin{cases}v_k,\quad \text{if}~j=i+k~\text{for some}~k\in\mathcal{I},\\0,\quad\text{otherwise},\end{cases}
\end{equation}
where $\mathcal{I}$ is a prescribed sparse subset of $\{-2^d,-2^d+1,\cdots,2^d\}$ and $\{v_k\}_{k\in\mathcal{I}}$ is a prescribed set of positive constants. Followed by column normalization on $\widetilde{\BT}$, we obtain the transition probability matrix $\BT$ whose each column adds up to 1. In practice, we casually choose $\mathcal{I}=\{-13,-5,2,6\}$ and $\{v_{13},v_{-5},v_2,v_6\}=\{1,4,3,2\}$.

The proposed method is implemented to find the principle eigenvector of $\BT$. First, we introduce the fictitious domain $[0,1]^d$ and take two grid points $1/3$ and $2/3$ in each dimension. Then we take the following penalized model to compute $(\BI-\BT)\Bu=0$,
\begin{equation}\label{19}
\min_{\theta}~\frac{1}{N^d}\|(\BI-\BT)\Bphi(\theta)\|_2^2+\varepsilon^{-1}\left(N^{-d}\cdot\bm{1}^\top\Bphi(\theta)-1\right)^2,
\end{equation}
where $\bm{1}$ is the all-ones column vector in $\mathbb{R}^{2^d}$; namely, we require the mean of the approximate solution to be 1.

This example is slightly different from previous ones. In each dimension, the solution of the linear system characterizes the Boolean state of one object, so the number of grid points in each dimension is always two. Therefore, the solution $\Bu$ is a grid function defined at a $2\times2\times\cdots\times2$ ($d$ times) grid, which can be interpolated by tensor product linear polynomials. Indeed, recall the tensor product polynomial space of degree $1$ is given by $\mathcal{P}:=\text{span}\{x_1^{p_1}x_2^{p_2}\ldots x_d^{p_d}: p_i=0~\text{or}~1~\text{for}~i=1,\ldots,d\}$. Since the degree of freedom of $\mathcal{P}$ is exactly $2^d$, there exists a unique polynomial $f$ in $\mathcal{P}$ such that $u_\Balpha=f(\Bx_\Balpha)$ for all $\Balpha\in\Lambda$. Moreover, $f$ is analytic and smooth in $\overline{\Omega}$ and hence can be extended to a Barron function on $\mathbb{R}^d$ with a relatively small Barron norm. So we can expect that $\|\Bu\|_{\mathcal{B},\Gamma}$ in the error bound \eqref{21} is small.

The model \eqref{19} is implemented using a mini-batch gradient descent algorithm with $\varepsilon=1.0$, $|\mathcal{S}|=2\times10^4$ and number of iteration $2\times10^4$. In each iteration, the term $N^{-d}\cdot\bm{1}^\top\Bphi(\theta)$ is only calculated at the training points; specifically, we calculate $|\mathcal{S}|^{-1}\sum_{x\in\mathcal{S}}\phi(x;\theta)$ instead. The cases $d=50$ and $100$ are tested, where the matrix $\BI-\BT$ has $O(10^{15})$ and $O(10^{30})$ nonzero entries, respectively. The resulting residuals are listed in Table \ref{Tab_case4_res}. We remark that our experiment computes much larger systems than the previous work \cite{Li2012}, which solves the same problem with at most 30 dimensions and $5\times10^4$ nonzero entries.

\begin{table}[h!]\small
\centering
\begin{tabular}{l|cc}
  \toprule
 $(L,M)$ & $d=50$ & $d=100$ \\\hline
$(2,100)$& 6.146e-04 $\pm$ 1.286e-04& 1.783e-02 $\pm$ 1.279e-03\\
$(2,200)$& 5.809e-04 $\pm$ 8.542e-05& 1.761e-02 $\pm$ 8.862e-04\\
$(3,100)$& 5.239e-04 $\pm$ 8.815e-05& 3.504e-03 $\pm$ 8.978e-04\\
$(3,200)$& 5.031e-04 $\pm$ 8.915e-05& 3.840e-03 $\pm$ 2.047e-04\\\bottomrule
\end{tabular}
\caption{\em Residuals for various $d$, $L$ and $M$ in the probabilistic Boolean network problem .}
\label{Tab_case4_res}
\end{table}

\section{Conclusion}
This work develops a novel NN-based method for extremely large linear systems. The main advantage lies in the saving of storage. Specifically, we create a neural network representation for the unknown vector, containing much fewer free elements than the original linear system. The system is then modified to a nonlinear least-squares optimization, and it can be solved by gradient descent under a deep learning framework. The proposed method allows us to deal with problems out of storage if using traditional linear solvers. An error estimate is also provided using the approximation property of NNs.

Several physical problems are considered in the numerical experiments. This method is successfully implemented to solve the corresponding linear systems. Compared with prior work on these problems, we solve systems of much larger sizes, usually intractable for other existing methods. However, the accuracy of this NN-based method is generally not as high as traditional ones due to the optimization error. Hence it is not recommended for small linear systems.

Moreover, as mentioned in Section \ref{sec:remark_Poisson}, the effectiveness of the method relies on the smoothness of the physical problem. It is required that the solution and the right hand side of the linear system do not oscillate globally or locally. Otherwise, the proposed method will fail. For example, suppose the physical solution is zero except for a few localized spikes in small regions. In that case, the corresponding right hand side of the linear system will also be zero but a small number of components. So with high probability, the equations selected as the training set will have exactly zero right hand sides, and the numerical solution is identically zero. In other words, this method cannot capture the very local property of the solution that is far away from the global tendency.

One direction of future work could be the convergence analysis of the gradient descent in solving the least-squares optimization. Namely, we could investigate whether the gradient descent necessarily finds good minimizers. In recent years, some research work has been conducted on the convergence of gradient descent in neural network regression, yet it is significantly different from this situation. On the one hand, the loss function in this method is the residual of the linear system rather than the simple $\ell^2$ or entropy loss discussed in prior work. On the other hand, most of the previous analysis is based on the over-parametrization hypothesis, in which the NN has much more parameters than terms in the loss function. But in this method, we expect to use NNs with much fewer parameters than equations or unknowns to save the storage.

Numerical results in Section 4.1.2 demonstrate that the accuracy and efficiency of the method are hardly affected by the degree of discretization $N$. In fact, as $N$ increases, although the linear system becomes larger, the mini-batch gradient descent will not carry more burden because the computational amount only depends on the batch size and the number of iterations. Also, we note that our method aims to learn a ``smooth" solution, which is a discretization of the physical solution of the original continuous problem, instead of unstructured discrete data. Hence we do not need a very wide or deep (i.e., over-parametrized) NN as the learner of the target solution. Recent literature \cite{Andoni2014,Zhu2019} implies that in learning problems, if the target function is smooth enough (e.g., a polynomial or an NN-like function), it suffices to use a small learner network whose size does not increase with the number of training samples, and it can be successfully trained by gradient descent. Therefore, the analysis of the gradient descent optimization in our method could be made in similar ways, avoiding the over-parametrization framework.

We are also inspired by the last numerical example, in which the solution represents binary probability distribution and does not characterize any smooth physical quantities. One open question is whether the smoothness hypothesis of the solution is necessary for the success of this method if we regard the linear system as a $2^d\times2^d$ structure. In this case, the approximate network only has to fit two points in every dimension. It is simply required that the network acts as a straight line in any dimension. Therefore it is interesting to investigate whether general $2^d\times2^d$ linear systems can be handled by this method without many hypotheses on the solution.

\bibliography{expbib}
\bibliographystyle{plain}

\appendix

\end{document}